\documentclass[12pt]{article}
\usepackage[hmargin={1.75truecm,1.90truecm},vmargin={2truecm,2truecm}]{geometry}
\usepackage[utf8]{inputenc}
\usepackage{hyperref}
\usepackage{amsmath}
\usepackage{amsthm}
\usepackage{amsfonts}
\usepackage{amssymb}
\usepackage{graphicx}
\usepackage{xcolor}
\usepackage{arydshln}
\usepackage[affil-it]{authblk}
\usepackage{natbib}
\usepackage{url}
\usepackage{afterpage}
\usepackage{amsmath} 
\usepackage{amssymb} 
\usepackage{multirow}
\usepackage{multicol}
\usepackage{booktabs}
\usepackage{orcidlink}
\usepackage[linesnumbered,ruled,vlined]{algorithm2e}

\SetCommentSty{AlgoCommentStyle}
\usepackage{rotating} 
\usepackage[linesnumbered,ruled]{algorithm2e} 
\usepackage{amsthm}
\usepackage{caption}
\usepackage{subcaption}
\usepackage{cleveref}
\usepackage{ragged2e}
\usepackage{todonotes}
\usepackage{tikz}
\usepackage{tikz-3dplot}
\usepackage{tikzscale}
\usepackage{pdflscape} 
\usepackage{float}
\usetikzlibrary{arrows}
\usetikzlibrary{shapes}
\usetikzlibrary{snakes}
\usetikzlibrary{patterns}
\usetikzlibrary{backgrounds,topaths}   
\usepackage{lipsum}
\usepackage{threeparttable}

\newtheorem{theorem}{Theorem}[section]
\newtheorem{example}[theorem]{Example}
\newtheorem{observation}[theorem]{Observation}
\newtheorem*{observation*}{Observation}
\newtheorem{remark}[theorem]{Remark}
\newtheorem{proposition}[theorem]{Proposition}

\newtheorem{corollary}[theorem]{Corollary}

\crefname{theorem}{Theorem}{Theorems}
\crefname{example}{Example}{Examples}
\crefname{observation}{Observation}{Observations}
\crefname{remark}{Remark}{Remarks}
\crefname{proposition}{Proposition}{Propositions}
\crefname{lemma}{Lemma}{Lemmas}
\crefname{corollary}{Corollary}{Corollaries}
\crefname{algorithm}{Algorithm}{Algorithms}
\crefname{table}{Table}{Tables}	
\crefname{figure}{Figure}{Figures}
\crefname{section}{Section}{Sections}

\providecommand{\keywords}[1]{\small \quad \quad \textbf{Keywords:} #1}
\newcommand{\ROMAN}[1]{\uppercase\expandafter{\romannumeral#1}}
\newcommand{\rev}[1]{{\color{black}#1}}

\newcommand{\MCLP}{{MCLP}\xspace}
\newcommand{\MinCLP}{{MinCLP}\xspace}
\newcommand{\MCLPNW}{{GMCLP}\xspace}

\newcommand{\LP}{\text{LP}\xspace}
\newcommand{\LIN}{\text{\rm L}\xspace}

\newcommand{\CONR}{\text{CONS-REDUCTION}\xspace}

\newcommand{\CPX}{\texttt{CPX}\xspace}

\newcommand{\IDI}{\texttt{CPXC+IDT}\xspace}
\newcommand{\CPXC}{\texttt{CPXC}\xspace}
\newcommand{\BDC}{\texttt{BD}\xspace}

\newcommand{\NOAGG}{\texttt{NO\_AGG}\xspace}
\newcommand{\NODR}{\texttt{NO\_DR}\xspace}
\newcommand{\NOVI}{\texttt{NO\_TCI}\xspace}

\newcommand{\DeltaS}{\texttt{$\Delta$S}\xspace}
\newcommand{\RT}{\texttt{RT}\xspace}
\newcommand{\RN}{\texttt{RN}\xspace}
\newcommand{\RGPC}{\texttt{$\Delta$GPC}\xspace}
\newcommand{\RV}{\texttt{RV}\xspace}
\newcommand{\RC}{\texttt{RC}\xspace}

\newcommand{\solver}[1]{\textsc{#1}\xspace}
\newcommand{\CPLEX}{\solver{CPLEX}}

\newcommand{\MIP}{\text{MIP}\xspace}

\newcommand{\BD}{{BD}\xspace}

\newcommand{\CI}{\mathcal{I}}
\newcommand{\CJ}{\mathcal{J}}

\newcommand{\CP}{\mathcal{P}}

\newcommand{\CS}{\mathcal{S}}

\newcommand{\A}{\mathcal{A}}
\newcommand{\Y}{\mathcal{Y}}
\newcommand{\N}{\mathcal{N}}

\newcommand{\C}{\mathcal{C}}

\newcommand{\DATAONE}{T1\xspace}
\newcommand{\DATATWO}{T2\xspace}

\newcommand{\tblT}{\texttt{T}\xspace}

\newcommand{\tblN}{\texttt{N}\xspace}

\newcommand{\tblGPC}{\texttt{GI\,\%}\xspace}

\newcommand{\tblST}{\texttt{ST}\xspace}
\newcommand{\tblPT}{\texttt{PT}\xspace}
\newcommand{\tblDC}{\texttt{$\Delta$C}\xspace}
\newcommand{\tblFV}{\texttt{$\Delta$V}\xspace}

\newcommand{\tblLB}{$z$\xspace}

\newcommand{\tblS}{\texttt{S}\xspace}
\DeclareMathOperator*{\argmin}{argmin}
\DeclareMathOperator*{\argmax}{argmax}

\newcommand{\GMI}{{GMI}\xspace}
\newcommand{\UNIT}{\text{U-$0.5$}\xspace}
\newcommand{\NU}{\text{NU}\xspace}
\newcommand{\LPG}{\texttt{LPG\,\%}\xspace}
\newcommand{\CPXCTCI}{\texttt{CPXC+T}\xspace}

\title{Presolving and cutting planes for the generalized maximal covering location problem}
\author[a,b]{Wei Lv\orcidlink{0009-0009-6861-9532}}
\author[a]{Cheng-Yang Yu}
\author[a]{Jie Liang}
\author[a,b]{Wei-Kun Chen\orcidlink{0000-0003-4147-1346}}
\author[c,d]{Yu-Hong Dai\orcidlink{0000-0002-6932-9512}}

\affil[a]{\small School of Mathematics and Statistics, Beijing Institute of Technology, Beijing 100081, China\\\textit{\{lvwei,yuchengyang,liangjie,chenweikun\}@bit.edu.cn}}
\affil[b]{\small State Key Laboratory of Cryptology, P. O. Box 5159, Beijing, 100878, China}
\affil[c]{\small Academy of Mathematics and Systems Science, Chinese Academy of Sciences, Beijing 100190, China}
\affil[d]{\small School of Mathematical Sciences, University of Chinese Academy of Sciences, Beijing 100049, China\\\textit{dyh@lsec.cc.ac.cn}}
\date{}

\begin{document}

\maketitle

\begin{abstract}
	This paper considers the generalized maximal covering location problem (\MCLPNW) which establishes a fixed number of facilities to maximize the weighted sum of the covered customers, allowing customer weights to be  positive or negative.
	Due to the huge number of linear constraints to model the covering relations between the candidate facility locations and customers, and particularly the poor linear programming (\LP) relaxation,
	the \MCLPNW is extremely difficult to solve by state-of-the-art mixed integer programming (\MIP) solvers. 
	To improve the computational performance of \MIP-based approaches for solving \MCLPNW{s}, we propose customized presolving and cutting plane techniques, which are isomorphic aggregation, dominance reduction, and two-customer inequalities.
	The isomorphic aggregation and dominance reduction can not only reduce the problem size but also strengthen the \LP relaxation of the \MIP formulation of the \MCLPNW.
	The two-customer inequalities can be embedded into a branch-and-cut framework to further strengthen the \LP relaxation of the \MIP formulation on the fly. 
	By extensive computational experiments, we show that all three proposed techniques can substantially improve the capability of \MIP solvers in solving \MCLPNW{s}. 
	In particular, for a testbed of $40$ instances with identical numbers of customers and candidate facility locations in the literature, the proposed techniques enable us to provide optimal solutions for $13$ previously unsolved benchmark instances;
	for a testbed of $336$ instances where the number of customers is much larger than the number of candidate facility locations, the proposed techniques can turn most of them from intractable to easily solvable.
	\vspace{8pt} \\
\end{abstract}

\section{Introduction}

The maximal covering location problem (\MCLP), first proposed by \cite{Church1974}, is one of the fundamental discrete optimization problems and has been widely investigated in the literature.
Given a collection of customers and a collection of candidate facility locations associated
with a notion of \emph{coverage}, which specifies whether or not a customer can be covered by a candidate facility location, the \MCLP attempts to establish a fixed number of facilities to maximize the weighted sum of the covered customers.
The \MCLP arises in or serves as a building block in a wide variety of applications, including emergency medical services \citep{Adenso-Diaz1997,Degel2015}, forest fire detection \citep{Bao2015}, ecological monitoring and conservation \citep{Farahani2014,Martin2021}, bike sharing \citep{Muren2020}, disaster relief \citep{Iloglu2020,Alizadeh2021}, waste collection \citep{Fischer2023}, and transportation \citep{Bucarey2022}.
For a detailed discussion of the variants and applications of the \MCLP, we refer to the recent surveys \cite{Farahani2012,Murray2016,Garcia2019,Marianov2024} and the references therein.

In the classic \MCLP of \cite{Church1974}, customer weights are assumed to be positive. 
This is usually applicable in the context of establishing desirable facilities such as supermarkets, garages, banks, and police stations. 
The more customers covered, the better.
For problems with undesirable or obnoxious facilities such as nuclear power stations and prisons, customers do not wish to be covered. 
In such contexts, the minimal covering location problem (\MinCLP), investigated in \cite{Church1976,Murray1998,Church2022}, is applicable.  
The \MinCLP attempts to locate a fixed number of facilities while minimizing the weighted sum of the covered customers. 
As such, the \MinCLP can be seen as the \MCLP with negative weights of customers.
\cite{Berman1996,Berman2003,Plastria1999} studied a special case of the \MinCLP where only a single undesirable facility has to be located.
\cite{Berman2008} investigated the \MinCLP with distance constraints which enforce a minimum distance between any pair of facilities.
For other variants of the \MinCLP, we refer to \cite{Berman2016,Karatas2021,Church2022,Khatami2023} among many of them.

In this paper, we consider a generalized version of the \MCLP and \MinCLP, called the generalized covering location problem (\MCLPNW), where the weights of the customers are allowed to be positive or negative \citep{Berman2009,Berman2010}.
The \MCLPNW (with a mixture of positive and negative customer weights) arises in the context that facilities are undesirable or obnoxious to certain customers while offering beneficial services to others.
For example, if the facilities are factories, polluting industrial units, or sewage treatment plants, residential districts may wish them to be located farther away (i.e., not to be covered), while industrial customers would benefit from the  proximity \citep{Drezner1991,Maranas1994}.
The \MCLPNW is also suitable for modeling problems with a mixture of desirable and undesirable customers. 
Two examples for this are detailed as follows.
First, when locating stores in a city, low-crime areas within the stores' coverage radius may be regarded as desirable customers, while high-crime areas may be seen as undesirable customers, as the stores may have to pay high insurance fees or suffer from revenue losses due to thefts and robberies \citep{Berman2009}.
Second, in a competitive environment, opening new facilities to serve many customers with positive demand is beneficial to revenue, but the proximity of competitors' facilities (i.e., undesirable customers) could decrease the expected profit \citep{Fomin2022}.

\cite{Berman2009} first generalized the mixed integer programming (\MIP) formulation of the classic \MCLP \citep{Church1974} and proposed an \MIP formulation for the \MCLPNW. 
Although this enables general-purpose \MIP solvers to find an optimal solution for the problem, 
solving the \MIP formulation of the \MCLPNW is very challenging for state-of-the-art \MIP solvers \citep{Berman2009,Berman2010}; 
for a testbed of $40$ instances with up to 900 candidate facility locations and customers, \cite{Berman2009} observed that only $21$ instances were solved to optimality by the \MIP solver \CPLEX within $2$ hours.

\subsection{Contributions and outlines}

The main motivation of this paper is to develop customized \MIP techniques to improve the computational performance of \MIP-based approaches for solving \MCLPNW{s}.
In particular, we first show that the presence of negative customer weights in the \MCLPNW could not only lead to 
a huge number of linear constraints to model the covering relations between the candidate facility locations and customers
but also result in an extremely poor linear programming (\LP) relaxation of the \MIP formulation of \cite{Berman2009}, thereby making state-of-the-art \MIP-based approaches (including calling \MIP solvers) inefficient to solve the \MCLPNW.
In an attempt to address these two challenges,
we then propose customized presolving and cutting plane techniques taking the special problem structure of the \MCLPNW into consideration.
To the best of our knowledge, this is the first time that customized \MIP techniques are developed to solve the \MCLP with (some or all) negative customer weights. 
The main contributions of this paper are summarized as follows.
\begin{itemize}
	\item We propose two customized presolving techniques, namely, isomorphic aggregation and dominance reduction. 
	The isomorphic aggregation aggregates several customers, covered by the same candidate facility locations, into a single customer.
	The dominance reduction derives a dominance relation between each pair of customers satisfying the condition that the candidate facility locations that can cover one customer can also cover the other.
	The presence of these dominance relations enables us to remove some constraints from the \MIP formulation of the \MCLPNW.
	Although the two proposed presolving techniques are designed to reduce the problem size of the \MIP formulation of the \MCLPNW, 
	they can also effectively strengthen the \LP relaxation of the problem formulation, making the reduced problem much more computationally solvable.
	\item We develop a family of valid inequalities, called two-customer inequalities, for the \MCLPNW. 
	The proposed two-customer inequalities generalize the relations derived by the dominance reduction, and can be embedded in a branch-and-cut framework to further strengthen the \LP relaxation of the \MIP formulation on the fly. 
	We also analyze how the proposed two-customer inequalities improve the \LP relaxation of the \MIP formulation, which plays an important role in the design of the separation algorithm.
\end{itemize}
Extensive computational results demonstrate that the three proposed techniques can substantially improve the capability of \MIP solvers in solving \MCLPNW{s}. 
In particular, for a testbed of {$40$} instances with identical numbers of customers and candidate facility locations \citep{Berman2009}, the proposed techniques enable us to provide optimal solutions for $13$ previously unsolved benchmark instances;	
for a testbed of $336$ instances where the number of customers is much larger than the number of candidate facility locations \citep{Cordeau2019}, the proposed techniques can turn most of them from intractable to easily solvable.
Moreover, compared to an extension of the state-of-the-art Benders decomposition (\BD) in \cite{Cordeau2019}, our approach (using an \MIP solver with the three proposed techniques) is significantly more efficient.

It is worthwhile remarking that although the three proposed techniques are motivated by the \MCLPNW, they can also be applied to solve the variants of the \MCLPNW that consider other practical constraints on the facilities such as the distance constraints of the facilities \citep{Moon1984,Berman2008,Grubesic2012}.

The remainder of the paper is organized as follows.
\cref{sec:literature} reviews the relevant literature on the \MCLPNW.
\cref{sec:formweak} introduces the \MIP formulation of \cite{Berman2009} and discusses the challenges of using  \rev{\MIP-based approaches} to solve them.
Sections \ref{isomo}, \ref{sec:dominance}, and \ref{sec_2link} develop the isomorphic aggregation, dominance reduction, and two-customer inequalities for the \MCLPNW, respectively.
\cref{sect:num} presents the computational results.
Finally, \cref{sec:conclusion} draws the conclusions.

\subsection{Literature review}\label{sec:literature}

In this subsection, we review the relevant references on the solution algorithms for the \MCLPNW and its two special cases, the \MCLP and \MinCLP. 

For the \MCLP, researchers have developed various heuristics and exact algorithms.
Here, we only review the relevant exact algorithms for solving the \MCLP; see recent surveys \cite{Farahani2012,Murray2016,Garcia2019} for a detailed review of various heuristic algorithms.
\citet{Dwyer1981} developed an LP-based branch-and-bound algorithm for solving a special case of the MCLP where all customers have equal weights. 
Subsequently, \citet{Downs1996} proposed a Lagrangian-based branch-and-bound algorithm to solve the (general) \MCLP.
The authors reported results on \MCLP instances with up to $74$ candidate facility locations and $2241$ customers. 
Recently, \citet{Cordeau2019} developed a \BD to solve large-scale realistic \MCLP{s} where the number of customers is much larger than the number of candidate facility locations.
Their results demonstrated that the \BD is capable of solving \MCLP{s} with $100$ candidate facility locations and up to $15$ million customers.
\cite{Lamontagne2024} and \cite{Guney2020} used a similar \BD to solve \MCLP{s} in a dynamic setting and \MCLP{s} that are derived from influence maximization problems in social networks, respectively.
It is worthwhile remarking that the \LP relaxation of the standard \MIP formulation of the \MCLP is usually tight or near tight \citep{ReVelle1993,Snyder2011,Cordeau2019}, which enables \rev{state-of-the-art \MIP-based approaches} to solve moderate-sized instances to optimality within a reasonable period of time. 
\cite{Chen2023b} further proposed various customized presolving techniques to enhance the capability of  \rev{state-of-the-art \MIP-based approaches} in solving large-scale \MCLP{s}.
\rev{In \cref{sec:formweak}, we extend the presolving techniques of \cite{Chen2023b} to solving the \MCLPNW.}

In contrast to the \MCLP which can be easily tackled by state-of-the-art \MIP-based approaches (at least for moderate-sized instances), the presence of negative customer weights in the \MinCLP or \MCLPNW makes the problem extremely hard to solve by \MIP solvers. 
For the \MinCLP, \cite{Murray1998} observed that solving \MinCLP{s} by an \MIP solver requires a large computational effort; for instances with only $79$ candidate facility locations and customers, it requires up to $25$ nodes and $83$ seconds to find an optimal solution.
For a variant of the \MinCLP where the distance constraints are included, the results in \cite{Berman2008} show that \CPLEX even failed to solve instances with $500$ candidate facility locations and customers within the $1800$ seconds time limit.  
For the \MCLPNW, the results in \cite{Berman2009} reveal that it is inefficient to use \MIP solvers to find an optimal solution within a reasonable period of time. 
Despite such challenges, no customized \MIP technique for the \MCLPNW or its special case \MinCLP has been explored in the literature until now. 
\cite{Berman2008} developed three heuristic algorithms to find a feasible solution for their problem, which can also be used to solve the \MinCLP. 
\cite{Berman2009} designed the ascent algorithm, simulated annealing, and tabu search to find a feasible solution for the \MCLPNW.

\section{\MIP formulation and its weaknesses}\label{sec:formweak}
In this section, we will first review the \MIP formulation of \cite{Berman2009} for the \MCLPNW 
and then discuss the challenges to solve the formulation by \rev{\MIP-based approaches}.

\subsection{Problem formulation}\label{sec:formulation}
We start with the following notations for the \MCLPNW:
\begin{itemize}
	\item $\mathcal{I}$ and $i$: set and index of candidate facility locations;
	\item $\mathcal{J}$ and $j$: set and index of customers;
	\item $\mathcal{I}_j$: set of candidate facility locations that can cover customer $j$; 
	\item $w_j$: weight of customer $j$;
	\item $\N$: set of customers with negative weights $w_j < 0$;
	\item $p$: number of facilities to be established.
\end{itemize}
Usually,  
a customer $j$ can be covered by a candidate facility location $i$ if the distance $d_{ij}$ between $i$ and $j$ is less than or equal to a prespecified coverage distance $R$, and thus $\CI_j = \{i \in \CI\,:\, d_{ij} \leq R\}$.
We define the following two sets of binary variables:
\begin{equation*}
	\begin{aligned}
		y_i &=\left\{\begin{array}{ll}1,
			& {\text{if~facility}}~ i ~{\text{is open}};~\\
			0,& {\text{otherwise,}}\end{array}\right.&{\rm and}~~
		    x_j &=\left\{\begin{array}{ll}1,
			& {\text{if~customer}}~ j ~{\text{is covered}};~\\
			0,& {\text{otherwise}}.\end{array}\right.
	\end{aligned}
\end{equation*}
Throughout, for a vector $a \in \mathbb{R}^{n}$ and a subset $\CS\subseteq \{1,\dots, n\}$, we denote $a(\CS)= \sum_{i \in \CS}a_i$.
The \MCLPNW attempts to open $p$ facilities such that the weighted sum of the covered customers is maximized.
The \MIP formulation for the \MCLPNW \citep{Berman2009} can be written as:   
\begin{subequations}\label{mclp-nw1}
	\begin{align}
		\max \quad &\sum_{j\in \CJ}w_j x_j \nonumber\\
		\text{\rm s.t.}
		\quad &y(\CI) = p,\label{p-cover}\\
		\quad &y(\CI_j) \geq x_j, &&~\forall~ j \in \CJ \backslash \N, \label{pos-cover}\\
	    \quad &x_j \geq y_i, &&~\forall~j \in \N,~ i \in \CI_j,   \label{neg-cover} \\
		\quad &x_j \in \{0, 1\}, &&~\forall ~ j \in \CJ,\label{xbinary} \\
		\quad &y_i \in \{ 0, 1 \}, &&~\forall ~ i \in \CI.\label{ybinary}
	\end{align}
\end{subequations}
The objective function maximizes the weighted sum of the  covered customers.
Constraint \eqref{p-cover} ensures that the total number of open facilities is $p$.
The first family of covering constraints \eqref{pos-cover} guarantees that for each customer $j$ with a nonnegative weight $w_j \geq 0$, if it is covered, then at least one of the candidate facility locations in set $\CI_j$ must be open.
The second family of covering constraints \eqref{neg-cover} guarantees that for customer $j$ with a negative weight $w_j <0$, if there exists some open facility $i$ that can cover it, then it must be covered.
Finally, constraints \eqref{xbinary} and \eqref{ybinary} restrict the decision variables to be binary.

	\cite{Chen2023b} developed various presolving techniques to reduce the problem size and improve the efficiency of employing \MIP solvers in solving the classic \MCLP (i.e., formulation \eqref{mclp-nw1} with $\N = \varnothing$). 
	Four presolving techniques of \cite{Chen2023b} can also be adapted to the (general) \MCLPNW \footnote{Due to the equality constraint \eqref{p-cover} and the presence of customers $j$ with negative weights $w_j < 0$, the presolving technique (called domination) in \cite{Chen2023b} for the classic \MCLP cannot be applied to (the general) problem \eqref{mclp-nw1}.} and are summarized as follows.
	\begin{itemize}
		\item {P1:} If $\CI_j = \{i\}$ for some $i \in \CI$ and $j \in \CJ\backslash\N$, 
		then variable $x_j$ can be replaced by variable $y_i$ and constraint $y_i \geq x_j$ can be removed from formulation \eqref{mclp-nw1};
		\item {P2:} Given $j, r \in \CJ\backslash\N$, if $\CI_j = \CI_r$, then variable $x_r$ can be replaced by variable $x_j$ and constraint $y(\CI_r) \geq x_r$ can be removed from formulation \eqref{mclp-nw1};
		\item {P3:} Given $r, j_1, \ldots, j_\tau \in \CJ\backslash\N$ such that $\CI_{j_k} \subseteq \CI_r$ for all $k  = 1, 2, \ldots, \tau$ and $\CI_{j_{k_1}} \cap \CI_{j_{k_2}} = \varnothing$ for all $k_1, k_2 \in \{1,2,\ldots,\tau\}$ with $k_1 \neq k_2$, 
		constraint $y(\CI_r) \geq x_r$ can be replaced by constraint $\sum_{k=1}^{\tau} x_{j_k} + y(\CI_r\backslash\!\cup_{k=1}^{\tau} \CI_{j_k}) \geq x_r$;
		\item {P4:} For a node in the branch-and-cut search tree of solving formulation \eqref{mclp-nw1} by \MIP solvers, we can fix $y_i  = 0$ for all $i \in \CI_r$ and $r \in \CJ_0$, where $\CJ_0 \subseteq \CJ\backslash\N$ is the set of variables fixed at zero.
	\end{itemize}
	The derivations of the above presolving techniques for the \MCLPNW are similar to those in \cite{Chen2023b} and thus are omitted here.

\subsection{Challenges of solving the \MIP formulation \eqref{mclp-nw1}}\label{sec_weak}

Formulation \eqref{mclp-nw1} generalizes the well-known \MCLP \citep{Church1974} in which $\N= \varnothing$.
Although the \MCLP is NP-hard \citep{Megiddo1983}, state-of-the-art \MIP-based approaches can solve moderate-sized or even large-scale instances within a reasonable period of time \citep{Snyder2011,Cordeau2019,Chen2023b}.
However, for the \MCLPNW with some negative customer weights, solving the instances of formulation \eqref{mclp-nw1} by the current \MIP-based approaches is very challenging due to the following two weaknesses.

First, for a customer $j$ with a negative weight $w_j <0$, $|\CI_j|$ constraints $x_j \geq y_i$, $i \in \CI_j$, are required to model the covering relation between the candidate facility locations and customer $j$. 
This is intrinsically different from modeling the covering relation between the candidate facility locations and a customer with a nonnegative weight where only a single constraint $y(\CI_j) \geq x_j$ is needed. 
As such, compared with that of the classic \MCLP, the number of covering constraints in formulation \eqref{mclp-nw1} of the \MCLPNW is usually much larger, especially for the case with a large $|\N|$ or $|\CI_j|$, $j \in \N$.
The huge number of covering constraints makes it potentially much more expensive to solve even the \LP relaxation of formulation \eqref{mclp-nw1}, deteriorating the overall performance of \MIP solvers. 
Note that the aforementioned presolving techniques P1--P4 are not designed for problems with some negative customer weights, and their effectiveness in reducing the number of covering constraints of the \MCLPNW is limited, as observed in our experiments.
\begin{remark}
\cite{Berman2009} addressed the huge number of constraints \eqref{neg-cover} by replacing them with the aggregated constraints:
\begin{equation}\label{aggregated}
	y(\CI_j) \leq p x_j,~\rev{\forall~j \in \N}.
\end{equation}
Observe that when $x_j=0$, constraint \eqref{aggregated} also enforces $y_i=0$ for all $i \in \CI_j$; when $x_j =1$,  constraint \eqref{aggregated} is implied by \eqref{p-cover}.
However, replacing constraints \eqref{neg-cover} with the aggregated constraints in \eqref{aggregated} generally leads to a poor \LP relaxation.
In Appendix A of the online supplement\footnote{The online supplement is available at: \url{https://drive.google.com/file/d/1pRtDE26j48w3sJXMueR0MflnLWhI5F5Y/view?usp=share_link}.}, we observed that this operation does not improve the performance of solving formulation \eqref{mclp-nw1}.
Therefore, we will not consider the aggregated version of the covering constraints in the subsequent discussions.
\end{remark}

Second, unlike the classic \MCLP whose \LP relaxation is usually tight or near tight \citep{ReVelle1993,Snyder2011,Cordeau2019}, 
the presence of negative customer weights $w_j <0$, $j \in \N$, could lead to an extremely poor \LP relaxation, thereby forcing the branch-and-cut procedure to explore a huge number of nodes.
To see this, we first characterize the optimal value of formulation \eqref{mclp-nw1} and its \LP relaxation using the $y$ variables, which is based on the following observation.
\begin{observation}\label{keyobs}
	(i) There exists an optimal solution $(x^*,y^*)$ of formulation \eqref{mclp-nw1} such that
	\begin{align}\label{optd1}
		& x^*_j =  \min\{1,y^*(\CI_j) \} =\max_{i \in \CI_j} y^*_i,~\forall~ j \in \CJ.
	\end{align}
	(ii) There exists an optimal solution $(x^*,y^*)$ of the \LP relaxation of formulation \eqref{mclp-nw1} such that
	\begin{equation}\label{optd2}
		x^*_j =\left\{ \begin{array}{ll} 
			\max_{i \in \CI_j} y^*_i, & \text{if}~j \in \N; \\
			\min\{1,y^*(\CI_j) \},&\text{otherwise},
		\end{array} \right.
		\forall~j \in \CJ.
	\end{equation} 
\end{observation}

  \begin{theorem}\label{theo_opt}
  	Let $\Y = \left\{y \in \{0, 1\}^{|\CI|} \,:\, y(\CI) = p\right\}~\text{and}~\Y_{\text{\rm\LIN}} = \left\{y \in [0, 1]^{|\CI|} \,:\,  y(\CI) = p\right\}$.
The optimal values of formulation \eqref{mclp-nw1} and its \LP relaxation are given by
	\begin{align}
		& z = \max_{y\in \Y}\left\{\sum_{j\in \CJ} w_j \cdot \min \{1, y(\CI_j) \}\right\},\label{MCLPNW_O}\\
		& z_{\text{\rm\LP}} = \max_{y\in \Y_{\text{\rm\LIN}}}\left\{\sum_{j\in \N}w_j \cdot \max_{i \in \CI_j}y_i+
		\sum_{j\in \CJ\backslash \N} w_j \cdot \min \{1,  y(\CI_j) \} \right\}\label{MCLPNW_LP}.
	\end{align}
\end{theorem}
\noindent Compared with $z$ in \eqref{MCLPNW_O}, its upper bound $z_{\LP}$  in \eqref{MCLPNW_LP} is generally much larger; see \cref{sec:effIDT} further ahead.
Indeed, in contrast to the case with an integral point $y \in \Y$ where $ \min\{1,y(\CI_j) \} =\max_{i \in \CI_j} y_i$ holds for all $j \in \N$, for the case with a fractional point $y \in \Y_{\LIN}$, the term $\min\{1,y(\CI_j) \} $ could be much larger than the term $\max_{i \in \CI_j} y_i$ for $j \in \N$.
Hence, for a fractional point $y \in \Y_{\LIN}$, the objective value $\sum_{j\in \CJ} w_j \cdot \min \{1, y(\CI_j) \}$ of problem \eqref{MCLPNW_O} could be much smaller than the \rev{objective value} $\sum_{j\in \N}w_j \cdot \max_{i \in \CI_j}y_i+
\sum_{j\in \CJ\backslash \N} w_j \cdot \min \{1,  y(\CI_j) \}$ of problem \eqref{MCLPNW_LP} (as $w_j < 0$ for $j \in \N$), leading to a poor \LP relaxation bound $z_{\LP}$.  
The following example further illustrates this weakness.
\begin{example}\label{exam_lp}
	Consider a toy example of the \MCLPNW with $p=1$.  
	There are two customers that can potentially be covered by all candidate facility locations in $\CI$. 
	The two customers have weights $\frac{|\CI|+1}{|\CI|}$ and $-1$, respectively. 
	For this example, formulation \eqref{mclp-nw1} can be expressed as follows:
	\begin{equation}\label{MCLPNW_exm}
		z = \max_{(x, y)\in \{0, 1\}^2 \times \{0, 1\}^{|\CI|}} \left\{\frac{|\CI|+1}{|\CI|}x_1-x_2\, :\, y(\CI) = 1, \, y(\CI) \geq x_1,~x_2 \geq y_i, \, \forall~i \in \CI\right\}.
	\end{equation}
By \cref{theo_opt}, problem \eqref{MCLPNW_exm} and its \LP relaxation reduce to 
	\begin{align}
		& z = \max_{y\in \{0, 1\}^{|\CI|}} \left\{ \frac{|\CI|+1}{|\CI|} \min\{1, y(\CI) \}-\min\{1, y(\CI) \} \,:\, y(\CI) = 1 \right\},	\label{exm-ip}\\
		& z_{\text{\rm\LP}} = \max_{y\in [0, 1]^{|\CI|}}\left\{ \frac{|\CI|+1}{|\CI|} \min\{1, y(\CI) \}-\max_{i \in \CI} y_i \,:\, y(\CI) = 1\right\}.\label{ex-lp}
	\end{align}
	It is easy to see that (i) $z=\frac{1}{|\mathcal{I}|}$ where an optimal solution of \eqref{exm-ip} could be $\hat{y}=(1, 0, \ldots, 0)$; 
	and (ii) $z_{\text{\rm\LP}} =1$ where the only optimal solution of \eqref{ex-lp} is $\bar{y}=\left(\frac{1}{|\CI|}, \frac{1}{|\CI|}, \ldots, \frac{1}{|\CI|}\right)$.
	Thus,  when $|\CI|\rightarrow +\infty$, $\max_{i \in \CI} \bar{y}_i =\frac{1}{|\CI|} \ll 1 = \min\{1, \bar{y}(\CI) \}$, and 
	$\frac{z_{\text{\rm \LP}}}{z}=|\CI|$ goes to infinity.
	This example shows that in a very special and simple case,  the integrality gap of the \LP relaxation of formulation \eqref{mclp-nw1} could be infinity.
\end{example}
\begin{remark}
	Similar to the classic \MCLP, 
	\begin{equation}
	z_{\text{\rm R}} = \max_{y\in \Y_{\LIN}}\left\{\sum_{j\in \CJ} w_j \cdot \min \{1, y(\CI_j) \}\right\}\label{Eq:tmp}
	\end{equation}
	can also provide an upper bound for problem \eqref{MCLPNW_O}, which is tighter than $z_{\text{\rm \LP}}$ given in \eqref{MCLPNW_LP}.
	Unfortunately, \rev{unlike} $z_{\text{\rm \LP}}$ which can be computed by solving a polynomial-time compact \LP problem {\rm(}i.e., the \LP relaxation of formulation \eqref{mclp-nw1}{\rm)}, the computation for $z_{\text{\rm R}}$ is difficult. 
	In particular, it is unclear whether with the presence of negative customer weights $w_j$, $j \in \N$, problem \eqref{Eq:tmp} can still be represented as a compact \LP problem.
\end{remark}

It is well known that state-of-the-art \MIP solvers employ the branch-and-cut algorithmic framework, which implements various valid inequalities such as clique \citep{Atamturk2000}, zero-half \citep{Caprara1996}, and Gomory mixed integer (\GMI) inequalities \citep{Gomory1960} to strengthen the \LP relaxation of the \MIP problems.
However, as shown in \cref{sec:effIDT}, these inequalities, although valid for general \MIP problems, cannot effectively strengthen the \LP relaxation of the \MIP formulation of the \MCLPNW\footnote{In Appendix B of the online supplement, we present the computational results of using the recent learning-based \GMI inequalities of \cite{Chetelat2023} to solve the \MCLPNW. However, the results show that the learning-based \GMI inequalities, as other valid inequalities for general \MIP problems, also cannot effectively strengthen the \LP relaxation of the \MIP formulation of the \MCLPNW.}.

In summary, the presence of negative customer weights $w_j < 0$, $j \in \N$, could lead to a large problem size and a poor \LP relaxation, thereby \rev{making  state-of-the-art \MIP-based approaches inefficient to solve formulation \eqref{mclp-nw1}}. 
In the following three sections, we will develop customized presolving methods and cutting planes to overcome these two weaknesses.

\section{Isomorphic aggregation}\label{isomo}
Two customers $j$ and $r$ are called \emph{isomorphic} if they can be covered by the same candidate facility locations (i.e., $\CI_j = \CI_r$).
For two isomorphic customers $j$ and $r$, 
from \cref{keyobs}, there must exist an optimal solution $(x^*, y^*)$ of formulation \eqref{mclp-nw1} such that 
\begin{equation*}
			x^*_j = \min \{1,  y^*(\CI_j)\}
		~\text{and}~x^*_r = \min \{1, y^*(\CI_r) \}.
\end{equation*}
Then, it follows from $\CI_j = \CI_r$ that $x^*_j = x^*_r$.
Using this argument, we obtain
\begin{remark}\label{iso_remark}
	If $\CI_j = \CI_r$ holds for some distinct $j$ and $r$, then setting 
	$x_j = x_r$ does not change the optimal value of formulation \eqref{mclp-nw1}.
\end{remark}

By \cref{iso_remark}, we can remove variable  $x_r$ (or $x_j$) and the related constraints from formulation \eqref{mclp-nw1}.
This enables us to derive a presolving method, called isomorphic aggregation, to reduce the problem size of formulation \eqref{mclp-nw1}.
Let $\CI_{j_1}, \CI_{j_2}, \cdots, \CI_{j_\tau}$ be all distinct sets  in $\{\CI_j\}_{j \in \CJ}$, where $j_1, j_2,\ldots, j_\tau \in \CJ$ and $\tau \in \mathbb{Z}_+$. 
For each $k \in \CJ':= \{j_1, j_2, \ldots, j_{\tau}\}$, define $ \CJ_k := \{j \in \CJ \,:\, \CI_j = \CI_k\}$. 
By definition, the sets $\CJ_k$, $k \in \CJ'$, form a partition of $\CJ$.
After applying the isomorphic aggregation, there only exist $|\CJ'|$ customers in the (equivalently) reduced problem and each customer $k\in \CJ'$ has a weight $w'_k := w(\CJ_k)$.

\rev{The isomorphic aggregation generalizes the presolving technique P2 in \cref{sec:formulation} which only considers the aggregation of isomorphic customers with nonnegative weights. 
For the classic \MCLP \citep{Church1974} where all customers have nonnegative weights, the isomorphic aggregation has been shown to effectively reduce the problem size and improve the solution efficiency \citep{Chen2023b}.}
However, to the best of our knowledge, a detailed analysis of how the isomorphic aggregation affects the \LP relaxation is missing in the literature (even for the classic \MCLP).
In the following, we will analyze how this presolving method improves the \LP relaxation of the \MIP formulation \eqref{mclp-nw1} of the \MCLPNW.

Let $\N' \subseteq \CJ'$ be the set of customers  with a negative weight.
Since the formulation of the reduced problem is still a form of \eqref{mclp-nw1}, by \cref{theo_opt}, the relaxation of the reduced \MCLPNW reads
\begin{equation}\label{MCLPNW_LP2}
	z'_{\text{\rm\LP}} = \max_{y\in \Y_{\text{\rm\LIN}}}\left\{\sum_{k\in \N'}w'_k \cdot \max_{i \in \CI_k}y_i+
	\sum_{k\in \CJ'\backslash \N'} w'_k \cdot \min\left\{1,  y(\CI_k)\right\} \right\}.
\end{equation}
Let 
\begin{align}
	&z (y) =\sum_{j\in \N}w_j \cdot \max_{i \in \CI_j}y_i +
	\sum_{j\in \CJ\backslash \N} w_j \cdot \min\{1, \ y(\CI_j)\} 
	,\label{objz}\\
	& z'(y) =\sum_{k\in \N'} {w}'_k \cdot \max_{i \in \CI_k	}y_i+
	 \sum_{k\in \CJ'\backslash \N'} {w}'_k \cdot \min\{1, y(\CI_k)\} , \label{objz0}
\end{align}
be the objective functions of problems \eqref{MCLPNW_LP} and \eqref{MCLPNW_LP2}, respectively,
and let
\begin{equation*}
	\CP_{k} =\CJ_k \backslash \N~\text{for}~k \in \N'~\text{and}~\N_k = \CJ_k\cap \N ~\text{for}~k \in \CJ'\backslash \N'.
\end{equation*}
By the above definitions,
the customers in $\CP_k$, $k \in \N'$, have nonnegative weights (in the original problem) but will be aggregated to a customer with a negative weight (in the reduced problem); and 
the customers in $\N_k$, $k \in \CJ'\backslash \N'$, have negative weights (in the original problem)  but will be aggregated to a customer with a nonnegative weight (in the reduced problem).
To characterize how the  isomorphic aggregation improves the \LP relaxation bound, we need the following result.
\begin{theorem}\label{keyiso_thm}
	Let $y \in \Y_{\text{\rm \LIN}}$ and $f_k(y) = \min\{ 1, y(\CI_k) \} - \max_{i \in \CI_k}y_i$, $k \in \CJ'$.
	Then $f_k (y) \geq 0$ for $k \in \CJ'$ and
	\begin{equation}\label{diffz}
		z(y)  -z'(y)= \sum_{k \in \N'}|w(\CP_k)| f_{k}(y)+ \sum_{k \in \CJ' \backslash \N'} |w(\N_k)| f_{k}(y)\geq 0.
	\end{equation}
\end{theorem}
\begin{proof}
	By $y \in \Y_{\LIN}$, we have $y\in[0,1]^{|\CI|}$ and thus $f_{k}(y) \geq 0$, $k \in \CJ'$.
	For $k \in \N'$, using $w'_k = \sum_{j \in \CJ_k}w_j$ and $\CI_j = \CI_k$ for $j \in \CJ_k$, we obtain
	\begin{equation}\label{tmpeq1}
		w'_k \cdot \max_{i \in \CI_k} y_i = \sum_{j \in \CJ_k} w_j \cdot \max_{i \in \CI_j} y_i = \sum_{j \in\CJ_k \backslash \CP_k} w_j \cdot \max_{i \in \CI_j} y_i + \sum_{j \in \CP_k} w_j \cdot \max_{i \in \CI_j} y_i.
	\end{equation}
	Similarly, for $k\in \CJ'\backslash \N'$, we have
	\begin{equation}\label{tmpeq2}
		w'_k \cdot \min\{1, y(\CI_k)\}=\sum_{j \in \CJ_k} w_j \cdot \min\{1, y(\CI_j)\}=   \sum_{j \in\N_k} w_j \cdot \min\{1, y(\CI_j)\} + \sum_{j \in \CJ_k\backslash\N_k}w_j \cdot  \min\{1, y(\CI_j)\}.
	\end{equation}
	Substituting \eqref{tmpeq1}--\eqref{tmpeq2} into \eqref{objz0} and using \eqref{objz}, we have
	\begin{equation*}
		\begin{aligned}
			z(y) - z'(y)
			= & \sum_{k \in \N'} \sum_{j \in \CP_k} w_j \cdot \left(\min\{ 1, y(\CI_j) \} - \max_{i \in \CI_j}y_i\right) - \sum_{k \in \CJ'\backslash \N'} \sum_{j \in \N_k} w_j \cdot \left(\min\{ 1, y(\CI_j) \} - \max_{i \in \CI_j}y_i\right) \\
			= & \sum_{k \in \N'} w(\CP_k) \left(\min\{ 1, y(\CI_k) \} - \max_{i \in \CI_k}y_i\right) - \sum_{k \in \CJ'\backslash \N'} w(\N_k) \left(\min\{ 1, y(\CI_k) \} - \max_{i \in \CI_k}y_i\right) \\
			= & \sum_{k \in \N'}w(\CP_k) f_{k}(y)- \sum_{k \in \CJ' \backslash \N'} w(\N_k) f_{k}(y) \\
			= & \sum_{k \in \N'}|w(\CP_k)| f_{k}(y)+ \sum_{k \in \CJ' \backslash \N'} |w(\N_k)| f_{k}(y)\geq 0.\qedhere
		\end{aligned}
	\end{equation*}
\end{proof}
Using \cref{keyiso_thm}, we can give conditions under which $z_{\LP}=z'_{\LP}$ holds.
Specifically, if $\N=\varnothing$, i.e., the case that all customers have nonnegative weights \citep{Church1974}, then it follows $\N_k=\varnothing$ for $k\in \CJ' \backslash \N'$ and $\N'=\varnothing$;
and if all customers have negative weights \citep{Church1976}, i.e., $\CJ\backslash\N=\varnothing$, then it follows $\CP_k=\varnothing$ for $k \in \N'$ and $\CJ'\backslash\N'=\varnothing$.
In both cases, it follows from \eqref{diffz} that $z(y)=z'(y)$ holds for all $y \in \Y_\LIN$.
As a result,  
\begin{corollary}\label{coro_isoagg}
	If $\N=\varnothing$ or $\CJ \backslash\N=\varnothing$, then $z_{\text{\rm \LP}}= z'_{\text{\rm \LP}}$, where $z_{\text{\rm \LP}}$ and  $z'_{\text{\rm \LP}}$ are defined in \eqref{MCLPNW_LP} and \eqref{MCLPNW_LP2}, respectively. 
\end{corollary}

Using \cref{keyiso_thm}, it is also possible to give conditions under which the  isomorphic aggregation can improve the \LP relaxation bound, as detailed in the following corollary.
\begin{corollary}\label{theo_isoagg}
	Let $z_{\text{\rm \LP}}$ and  $z'_{\text{\rm \LP}}$ be defined in \eqref{MCLPNW_LP} and \eqref{MCLPNW_LP2}, respectively, and $y^*$ be an optimal solution of problem \eqref{MCLPNW_LP2}. 
	Then 
	\begin{equation}
		z_\text{\rm \LP}  -z'_\text{\rm \LP} \geq  \sum_{k \in \N'}|w(\CP_k)| f_{k}(y^*)+\sum_{k \in \CJ' \backslash \N'} |w(\N_k)| f_{k}(y^*).
	\end{equation}
 	Moreover, if (i)  $|w(\CP_k)| > 0$ and $f_k(y^*) >0$ hold for some $k \in \N'$, or (ii) $|w(\N_k)| > 0$ and $f_k (y^*)> 0$ hold for some $k \in \CJ' \backslash \N'$, then $z_{\text{\rm \LP}} >
 	z'_{\text{\rm \LP}}$.
\end{corollary}

The following example further illustrates the strength of the isomorphic aggregation. 
\begin{example}[continued]\label{exm_isoagg}
	After applying the isomorphic aggregation to the problem \eqref{MCLPNW_exm} in \cref{exam_lp}, the two customers are aggregated into a single customer with a positive weight $\frac{1}{|\CI|}$, and the \LP relaxation \eqref{MCLPNW_LP2} of the reduced problem reads
	\begin{equation*}
		\begin{aligned}
			z'_{\text{\rm \LP}} = \max_{y \in [0,1]^{|\CI|}} \left\{ \frac{1}{|\CI|} \min\{1, y(\CI) \} \,:\, y(\CI) = 1 \right\}= \frac{1}{|\CI|}=z,
		\end{aligned}
	\end{equation*} 
	where $z$ is defined in \eqref{exm-ip}.
	Thus, in contrast to the \LP relaxation of the original problem where the integrality gap could be infinity {\rm(}as shown in \cref{exam_lp}{\rm)}, the \LP relaxation of the reduced problem is tight.
\end{example}

To summarize, applying the isomorphic aggregation to formulation \eqref{mclp-nw1} of the \MCLPNW, we can obtain an equivalent reduced formulation that not only enjoys a smaller problem size (as the number of customers could become smaller) but also provides a potentially much stronger \LP relaxation (as shown in \cref{theo_isoagg}). 
These two advantages could make the reduced formulation much more computationally solvable by general-purpose \MIP solvers, as will be demonstrated in \cref{sect:num}.

\section{Dominance reduction}\label{sec:dominance}
Next, we derive a presolving method, called dominance reduction, by considering the \emph{dominance relations} between the customers.
A customer $j$ is dominated by a customer $r$ if $\CI_j \subseteq \CI_r$ (i.e., the candidate facility locations that can cover customer $j$ can also cover customer $r$).
Let $\A := \{(j, r) \,:\, j, r \in \CJ~\text{with}~j \neq r~\text{and}~\CI_j \subseteq \CI_r\}$ be the set of all dominance pairs.
For a dominance pair $(j, r) \in \A$, it follows from \cref{keyobs} that there must exist an optimal solution $(x^*, y^*)$ of formulation \eqref{mclp-nw1} such that 
\begin{equation*}
	x^*_j = \min \{1,  y^*(\CI_j)\}~\text{and}~x^*_r = \min \{1, y^*(\CI_r)\},
\end{equation*}
and by $\CI_j \subseteq \CI_r$, we must have $x^*_j \leq x^*_r$. 
Using the above argument, the dominance inequalities 
\begin{equation}\label{defA}
	x_j \leq x_r,~\forall~(j, r) \in \A,
\end{equation}
must be valid for formulation \eqref{mclp-nw1} in the sense that adding it into the formulation does not change the optimal value.
\begin{remark}
Formulation \eqref{mclp-nw1} is equivalent to
\begin{equation}\label{model:domineq}
	\max \left\{\sum_{j\in \CJ}w_j x_j \,:\, \eqref{p-cover}-\eqref{ybinary},~x_j \leq x_r,~\forall~(j, r) \in \A \right\}.
\end{equation}
\end{remark}

Note that if $\CI_j = \CI_r$, then the two dominance inequalities $x_j \leq x_r$ and $x_r \leq x_j$ imply $x_j = x_r$, and therefore, the \LP relaxation of problem \eqref{model:domineq} is at least as strong as the \LP relaxation of the reduced problem returned by the isomorphic aggregation (i.e., problem \eqref{MCLPNW_LP2}).
In the following, we shall show that how the dominance inequalities can be used to further (i) strengthen the \LP relaxation of the formulation \eqref{mclp-nw1} and (ii) perform reductions on removing some constraints from formulation \eqref{mclp-nw1}.

\subsection{Strengthening the \LP relaxation}\label{sec:slr}

Let 
\begin{equation}\label{cons:posneg}
	x_j \leq x_r,~\forall~(j, r) \in {\A}^{+-} := \{(j, r) \in \A \,:\, j \in \CJ\backslash\N, r \in \N\},
\end{equation} 
be a subset of the dominance inequalities in \eqref{defA}.
In other words, each inequality in \eqref{cons:posneg} corresponds to a dominance pair $(j, r)$, where $j$ is a customer with a nonnegative weight and $r$ is a customer with a negative weight.
We first demonstrate that in order to use the dominance inequalities in \eqref{defA} to strengthen the \LP relaxation of formulation \eqref{mclp-nw1}, only those in \eqref{cons:posneg} are needed.

To proceed, consider the problem 
\begin{equation}\label{model:domineq1}
	\max \left\{\sum_{j\in \CJ}w_j x_j \,:\, \eqref{p-cover}-\eqref{ybinary},~x_j \leq x_r,~\forall~(j, r) \in{\A}^{+-} \right\}
\end{equation}
and let $(x^*, y^*)$ be an optimal solution of its \LP relaxation.
Define 
\begin{align}
	& p_j = \argmax\{x^*_s \,:\,  s \in \CP(j) \}~\text{where}~\CP(j) = \{s \in \CJ\backslash\N \,:\, (s, j) \in {\A}^{+-}\}~\text{for}~j \in \N \label{pj},\\
	& n_j = \argmin\{x^*_s \,:\,  s \in \N(j) \}~\text{where}~\N(j) = \{s \in \N \,:\, (j, s) \in{\A}^{+-}\}~\text{for}~j \in \CJ\backslash\N.	\label{nj}
\end{align}
If $\CP(j)=\varnothing$, we let $p_j = 0$ and $x^*_{p_j}=0$; and if $\N(j)=\varnothing$, we let $n_j=-1$ and $x^*_{n_j}=1$.
$p_j$ and $n_j$ indeed depend on $x^*$ but we omit this dependence for notations convenience.
Using the above definitions, $0\leq x^*_s \leq x_{p_j}^* \leq x^*_{j}$ holds for all $j \in \N$ and $s \in \CJ\backslash\N$ with $(s,j) \in \A^{+-}$ and $x^*_j \leq x^*_{n_j} \leq x^*_s\leq 1$ holds for all $j \in \CJ\backslash\N$ and $s \in \N$ with $(j,s) \in \A^{+-}$. This, together with the fact that $w_j < 0$ for all $j \in \N$ and $w_j \geq 0$ for all $j \in \CJ\backslash \N$, enables us to characterize the optimal solutions of the \LP relaxation of problem \eqref{model:domineq1}.
\begin{remark}\label{dominianceLemma}
	There exists an optimal solution $(x^*, y^*)$ of the \LP relaxation of problem \eqref{model:domineq1} such that
		\begin{equation}\label{opt2}
		x^*_j =\left\{ \begin{array}{ll} 
			\max\left\{\max_{i \in \CI_j} y^*_i, x^*_{p_j} \right\}, & \text{if}~j \in \N;\\
			\min\{y^*(\CI_j), x^*_{n_j} \},&\text{otherwise},
		 \end{array} \right.
		 \forall~j \in \CJ.
	\end{equation}
\end{remark}

\noindent The following theorem shows that problems \eqref{model:domineq} and  \eqref{model:domineq1} provide the same \LP relaxation bound.
\begin{theorem}\label{prop:addpnineq}
	The \LP relaxations of problems \eqref{model:domineq}  and \eqref{model:domineq1} are equivalent in terms of sharing the same optimal value.
\end{theorem}
\begin{proof}
	Let $o_1$ and $o_2$ be the optimal values of the \LP relaxations of problems \eqref{model:domineq} and \eqref{model:domineq1}, respectively. 
	Clearly, $o_1 \leq o_2$ holds. 
	To show $o_1 \geq o_2$, by \cref{dominianceLemma}, it suffices to show that for an optimal solution $(x^*, y^*)$ of the \LP relaxation of \eqref{model:domineq1} satisfying \eqref{opt2}, it follows 
	$x^*_j \leq x^*_r$ for all $(j, r) \in \A\backslash{\A}^{+-}$.
	We consider the following three cases separately.
	\begin{itemize}
		\item [(i)] $j, r \in \CJ\backslash\N$. 
		It follows from the definitions of $\N(j),~\N(r)$ in \eqref{nj} and $\CI_j \subseteq \CI_r$ that $\N(r) \subseteq \N(j)$, and by \eqref{nj},  $x^*_{n_j} \leq x^*_{n_r}$ holds. 
		Together with $y^*(\CI_j) \leq y^*(\CI_r)$, we obtain 
		\begin{equation*}
			x^*_j  = \min\left\{y^*(\CI_j), x^*_{n_j} \right\} \leq  \min\left\{y^*(\CI_j), x^*_{n_r} \right\} \leq \min\left\{y^*(\CI_r), x^*_{n_r}\right\} =   x^*_r.
		\end{equation*}
		\item [(ii)] $j, r \in \N$.
		It follows from the definitions of $\CP(j),~\CP(r)$ in \eqref{pj} and $\CI_{j} \subseteq \CI_{r}$  that $\CP(j) \subseteq \CP(r)$, and by 
		 \eqref{pj}, $x^*_{p_{j}} \leq x^*_{p_{r}}$ holds. 
		Together with $\max_{i \in \CI_j} y^*_i \leq \max_{i \in \CI_r} y^*_i$, we obtain 
		\begin{equation*}
			x^*_j  = \max\left\{\max_{i \in \CI_j} y^*_i, x^*_{p_{j}} \right\}  \leq  \max\left\{\max_{i \in \CI_j} y^*_i, x^*_{p_{r}} \right\}  \leq \max\left\{\max_{i \in \CI_r} y^*_i, x^*_{p_{r}} \right\}  =   x^*_r.
		\end{equation*}
		\item [(iii)] $j \in \N$ and $r \in \CJ\backslash\N$.
		Since $(j, r) \in \A$, or equivalently, $\CI_j \subseteq \CI_r$, we have $\max_{i \in \CI_j} y^*_i \leq \max_{i \in \CI_r} y^*_i \leq y^*(\CI_r)$. 
		Hence, to show 
		\begin{equation*}
			x^*_j  =  \max\left\{\max_{i \in \CI_j} y^*_i, x^*_{p_j} \right\} \leq \min\left\{y^*(\CI_r), x^*_{n_r} \right\}  =  x^*_r,
		\end{equation*}
		it suffices to prove $\max_{i \in \CI_j} y^*_i \leq x^*_{n_r}$, $ x^*_{p_j} \leq y^*(\CI_r)$, and $ x^*_{p_j} \leq x^*_{n_r}$. We further consider four subcases.
		\begin{itemize}
			\item [1)] $\CP(j)=\varnothing$ and $\N(r)=\varnothing$. In this case, $x_{p_j}^*=0$ and $x_{n_r}^*=1$, and thus $\max_{i \in \CI_j} y^*_i \leq x^*_{n_r}$, $ x^*_{p_j} \leq y^*(\CI_r)$, and $ x^*_{p_j} \leq x^*_{n_r}$ hold.
			\item [2)] $\CP(j)=\varnothing$ and $\N(r)\neq\varnothing$. In this case, $x_{p_j}^*=0$,  and thus $ x^*_{p_j} \leq y^*(\CI_r)$ and $ x^*_{p_j} \leq x^*_{n_r}$ hold. Since $n_r \in \N$, from \eqref{opt2}, we have $x^*_{n_r} \geq \max_{i \in \CI_{n_r}} y^*_i\geq \max_{i \in \CI_j} y^*_i $, where the last inequality follows from $\CI_j \subseteq \CI_r$ and $\CI_r \subseteq \CI_{n_r}$ (as $n_r \in \N(r)$).
			\item [3)] $\CP(j)\neq\varnothing$ and $\N(r)=\varnothing$. In this case, $x^*_{n_r}=1$, and thus $\max_{i \in \CI_j} y^*_i \leq x^*_{n_r}$ and $ x^*_{p_j} \leq x^*_{n_r}$ hold. 
			Since $p_j \in\CJ \backslash \N$, from \eqref{opt2}, we obtain $x_{p_j}^* \leq y^*(\CI_{p_j}) \leq y^*(\CI_r)$, where the last inequality follows from $\CI_j \subseteq \CI_r$ and $\CI_{p_j} \subseteq \CI_{j}$ (as $p_j \in \CP(j)$).
			\item [4)] $\CP(j)\neq\varnothing$ and $\N(r)\neq\varnothing$.
			As $p_j \in \CP(j)\subseteq \CJ\backslash \N$ and $n_r \in \N (r) \subseteq \N$, we have 
			$\CI_{p_j} \subseteq \CI_j$ and $\CI_r \subseteq \CI_{n_r}$, respectively, which
			together with $\CI_j \subseteq \CI_r$, implies
			$\CI_{p_j} \subseteq \CI_{n_r}$ and thus $(p_j, n_r) \in {\A}^{+-}$.
			Therefore, $x^*_{p_j} \leq x_{n_r}^*$ holds.
			The proofs of $\max_{i \in \CI_j} y^*_i \leq x^*_{n_r}$ and $ x^*_{p_j} \leq y^*(\CI_r)$ are similar to those of cases 2) and 3), respectively.
			\qedhere
		\end{itemize}
	\end{itemize}
\end{proof}

\cref{prop:addpnineq} shows that in order to use the dominance inequalities to strengthen the \LP relaxation of formulation \eqref{mclp-nw1}, it suffices to consider those in \eqref{cons:posneg}.
The following theorem further provides a lower bound for the improvement on the \LP relaxation bound by the dominance inequalities in \eqref{cons:posneg}.
\begin{theorem}\label{th:improvelp}
	Let  $(x^*, y^*)$ be an optimal solution of the \LP relaxation of \eqref{model:domineq1} satisfying \eqref{opt2} and $z'_{\text{\rm \LP}}$ be the corresponding objective value.
	Then, 
	\begin{equation}\label{diffdominanceLP}
		\begin{aligned}
			z_{\text{\rm \LP}} - z'_{\text{\rm \LP}}
			\geq \sum_{j\in \N}w_j \cdot \min\left\{0, \max_{i \in \CI_j} y^*_i - x^*_{p_j} \right\} + \sum_{j\in \CJ\backslash \N} w_j \cdot \max\left\{ \min\{1, y^*(\CI_j)\} - x^*_{n_j}, 0 \right\} 
			\geq  ~ 0,
		\end{aligned}
	\end{equation}
	where $z_{\text{\rm\LP}}$ is defined in \eqref{MCLPNW_LP}.
	Moreover,
	if (i) $\max_{i \in \CI_j} y^*_i < x^*_{p_j}$ for some $j \in \N$, or
	(ii) $x^*_{n_j} < \min\{1, y^*(\CI_j)\}$ and $w_j > 0$ for some $j \in \CJ \backslash \N$, 
	then $z_{\text{\rm\LP}} > z'_{\text{\rm\LP}}$.
\end{theorem}
\begin{proof}
	Clearly, $y^*$ is a feasible  solution of problem \eqref{MCLPNW_LP}, and thus 
	\begin{equation}\label{tmpeqa}
		z_{\LP}\geq   \sum_{j\in \N}w_j \cdot \max_{i \in \CI_j}y^*_i + 
		\sum_{j\in \CJ\backslash \N} w_j \cdot \min \{1,  y^*(\CI_j) \}.
	\end{equation}
	From \eqref{opt2}, we have
	\begin{equation}\label{tmpeqb}
		z'_{\LP}= \sum_{j\in \N}w_j \cdot \max\left\{\max_{i \in \CI_j} y^*_i, x^*_{p_j} \right\}+
		\sum_{j\in \CJ\backslash \N} w_j \cdot  \min\left\{y^*(\CI_j), x^*_{n_j}\right\}.
	\end{equation}
	Combining \eqref{tmpeqa} and \eqref{tmpeqb}, we obtain \eqref{diffdominanceLP}.
	The proof of the second part is obvious.
\end{proof}

We use the following example to show that the condition in \cref{th:improvelp} could be satisfied, and demonstrate the potential of the dominance inequalities \eqref{cons:posneg} in strengthening the \LP relaxation of formulation \eqref{mclp-nw1}.
\begin{example}\label{domineq}
	Consider an example of the \MCLPNW where $p = 1$ and there exist two customers and three candidate facility locations.
	The weights of the two customers are $w_1 = 1$ and $w_2 = -1$, and $\CI_1 = \{1, 2\}$ and $\CI_2 = \{1, 2, 3\}$.
	As $\CI_1 \subseteq \CI_2$, the \LP relaxation of \eqref{model:domineq1} reads
	\begin{equation*}
		z'_{\rm LP} = \max_{(x, y)\in [0, 1]^2 \times [0, 1]^3}\left\{x_1 - x_2 \,:\, y_1 +  y_2 + y_3 = 1,\,y_1+y_2 \geq x_1,\,x_2 \geq y_1,\,x_2 \geq y_2,\,x_2 \geq y_3,\,x_1 \leq x_2 \right\}.
	\end{equation*}
	It is simple to see that  $(x^*, y^*) = (\frac{1}{3}, \frac{1}{3}, \frac{1}{3}, \frac{1}{3}, \frac{1}{3})$ is an optimal solution with the objective value $0$.
	By $\max_{i \in \CI_2} y_i^*  - x^*_1 = 0$, $\min\{1, y^*(\CI_1)\} - x^*_2 = \frac{1}{3}$, $w_1 = 1$, and \cref{th:improvelp}, we have $z_{\rm LP} -z'_{\rm LP} \geq \frac{1}{3}$.
\end{example}

\subsection{Constraint reduction}\label{sec:cr}
Let 
\begin{equation}\label{cons:negneg}
	x_j \leq x_r,~\forall~(j, r) \in {\A}^{--} := \{(j, r) \in \A \,:\, j \in \N, ~r \in \N\backslash\{j\}\},
\end{equation}
be another subset of the dominance inequalities in \eqref{defA}.
Each inequality in \eqref{cons:negneg} corresponds to a dominance pair $(j, r)$ where both $j$ and $r$ are customers with negative weights.
Although the inequalities \eqref{cons:negneg} cannot further improve the \LP relaxation of problem \eqref{model:domineq1} (as shown in \cref{prop:addpnineq}), they still hold the potential of eliminating some constraints in \eqref{neg-cover} from the problem.
Indeed, considering  a dominance pair $(j, r) \in {\A}^{--}$, the constraints $x_r \geq y_i$ for $i \in \CI_j\,(\subseteq  \CI_r)$ are implied by constraints $x_j \leq x_r$ and $x_j \geq y_i$ for $i \in \CI_j$.
Therefore, we can add inequality $x_j \leq x_r$ into problem \eqref{model:domineq1} and remove constraints $x_r \geq y_i$ for $i \in \CI_j \subseteq  \CI_r$  from the problem (without weakening its \LP relaxation).

Although the above reduction technique can remove some constraints in \eqref{neg-cover} from problem \eqref{model:domineq1}, it also requires the addition of some inequalities in \eqref{cons:negneg}. 
Therefore, the following question immediately arises: how to choose the dominance inequalities \eqref{cons:negneg} to apply the constraint reduction technique such that the number of constraints in the reduced problem is minimized?
We refer to this problem as problem \CONR. 

\begin{proposition}\label{np-hard}
	Problem {\rm \CONR} is strongly NP-hard.
\end{proposition}
\begin{proof}
The proof can be found in Appendix C of the online supplement.
\end{proof}

\cref{np-hard} implies that unless P=NP, there does not exist a polynomial-time algorithm to select the dominance inequalities in \eqref{cons:negneg} to apply the constraint reduction such that the number of constraints in the reduced problem is minimized.
We therefore develop a heuristic algorithm to achieve a trade-off between the performance and the time complexity.
The idea of the proposed algorithm lies in the fact that for $r \in \CJ$, the subsets $\CI_j$ with more elements are more preferable to be chosen as they can eliminate more constraints of the form $x_r \geq y_i$ (when $\CI_j \subseteq \CI_r$).
To this end, for each $r \in \CJ$, we recursively examine subsets $\CI_j$ according to the descending order of their cardinalities, 
and add the dominance inequality $x_j \leq x_r$ into problem \eqref{model:domineq1} if $\CI_j \subseteq \CI_r$ and at least two constraints of the form $x_r \geq y_i$ can be deleted concurrently. 
This heuristic procedure is summarized in \cref{addnnineq} and the overall complexity is $\mathcal{O}(|\N| \sum_{j \in \N} | \CI_j|)$.

In summary, the dominance reduction uses the dominance inequalities $x_j \leq x_r$ with $(j,r) \in\A^{+-}$ to strengthen the \LP relaxation of formulation \eqref{mclp-nw1} and those with $ (j,r) \in \bar{\A}^{--}$ (constructed by \cref{addnnineq}) to eliminate some constraints in \eqref{neg-cover}.
It is worth remarking that some dominance inequalities $x_j \leq x_r$,  $(j, r) \in \A^{+-} \cup \bar{\A}^{--}$, may be redundant.
In particular, if $(j, r), (r, s), (j, s)\in \A^{+-}\cup \bar{\A}^{--}$, then the dominance inequality $x_j \leq x_s$ is implied by $x_j \leq x_r$ and $x_r \leq x_s$. 
In our implementation of the dominance reduction, only the nonredundant dominance inequalities in $x_j \leq x_r$,  $(j, r) \in \A^{+-}\cup \bar{\A}^{--}$, will be added into formulation \eqref{mclp-nw1}.

\begin{algorithm}[!h]
	\caption{A heuristic algorithm for performing the constraint reduction}\label{addnnineq}
	Initialize $\bar{\A}^{--} \leftarrow \varnothing$ and $\bar{\CI}_j\leftarrow \CI_j$, $j \in \N$\;
	Reorder $\CI_j$, $j \in \N$, such that  $|\CI_{1}| \geq \dots \geq |\CI_{|\N|}|$\;
	\For{$r \leftarrow 1, \ldots, |\N|$}{
		\For{$j \leftarrow r+1, \ldots,|\N|$}{
			\If{$\CI_j \subseteq \CI_r$ and $|\CI_j \cap \bar{\CI}_r| \geq 2$}{
				Delete constraints $x_r \geq y_i$ for $i \in \CI_j \cap \bar{\CI}_r$ and add inequality $x_j \leq x_r$ into problem \eqref{model:domineq1}\;
				Update $\bar{\CI}_r \leftarrow \bar{\CI}_r \backslash \CI_j$ and $\bar{\A}^{--}  \leftarrow \bar{\A}^{--}  \cup \{(j, r)\}$\;
			}
		}
	}
\end{algorithm}

\section{Two-customer inequalities}\label{sec_2link}

In this section, we first present a family of valid inequalities, called two-customer inequalities, for formulation \eqref{mclp-nw1}.
Then, we investigate how  two-customer inequalities improve the \LP relaxation of formulation \eqref{mclp-nw1}, which plays an important role in  the design of the separation algorithm for the considered inequalities.

\subsection{Derived inequalities}

We start with the following result demonstrating that using the optimality condition \eqref{optd1}, a relation between any two distinct customers can be derived.
\begin{proposition}\label{th:2link}
	Let
	$(x^*,y^*)$ be an optimal solution of formulation \eqref{mclp-nw1} satisfying \eqref{optd1} and $j, r\in \CJ$ with $j \neq r$. 
	Then $x^*_j \leq x^*_r + y^*(\CI_j\backslash \CI_r)$ holds.
\end{proposition}
\begin{proof}
	If $x^*_j \leq x^*_r$, then $x^*_j \leq x^*_r + y^*(\CI_j\backslash \CI_r)$ holds naturally. 
	Otherwise, it follows from $x^* \in \{0,1\}^{|\CJ|}$ that $x^*_j=1 $  and $x^*_r =  0$. Then, using \eqref{optd1}, we obtain $y^*(\CI_j) \geq 1$ and $y^*(\CI_r) = 0$. 
	Consequently,  we have $y^*( \CI_j \backslash \CI_r) \geq 1$, and $x^*_j \leq x^*_r + y^*(\CI_j\backslash \CI_r)$ also holds.
\end{proof}
\noindent \cref{th:2link} enables us to derive a family of inequalities, called \emph{two-customer inequalities}, 
\begin{equation}\label{valid1}
	x_j \leq x_r + y(\CI_j\backslash \CI_r),~\forall~j\in\CJ,~r \in \CJ\backslash \{j\},
\end{equation}
which are valid for formulation \eqref{mclp-nw1} in the sense that adding them into formulation \eqref{mclp-nw1} does not change the optimal value.
 
Notice that if $\CI_j \subseteq \CI_r$, inequality $	x_j \leq x_r + y(\CI_j\backslash \CI_r)$ reduces to the dominance inequality $x_j \leq x_r$, and thus the two-customer inequalities in \eqref{valid1} generalize the dominance inequalities in \eqref{cons:posneg}.
In \cref{ex:twolink} of the next subsection, we show that compared with the dominance inequalities in \eqref{cons:posneg}, the two-customer inequalities in \eqref{valid1} can further strengthen the \LP relaxation of formulation \eqref{mclp-nw1}.

\subsection{How two-customer inequalities strengthen the  \LP relaxation of formulation \eqref{mclp-nw1}}

As demonstrated in \cref{prop:addpnineq}, in order to use the dominance inequalities $x_j \leq x_r$ in \eqref{cons:posneg} to strengthen the \LP relaxation of formulation \eqref{mclp-nw1}, it suffices to consider those with $j \in \CJ \backslash \N$ and $r\in \N$.
This result can be extended to the two-customer inequalities \eqref{valid1} as well and is formally stated in the following theorem.

\begin{theorem}\label{prop:2link}
	Let
	\begin{align}
		& \max \left\{ \sum_{j\in \CJ}w_j x_j \,:\, \eqref{p-cover}-\eqref{ybinary},~x_j \leq x_r + y(\CI_j \backslash \CI_r) ,~\forall~j, r \in \CJ~\text{with}~j \neq r \right\},	\label{2linkall}\\
		& \max \left\{\sum_{j\in \CJ}w_j x_j \,:\, \eqref{p-cover}-\eqref{ybinary},~x_j \leq x_r + y(\CI_j \backslash \CI_r), ~\forall~j \in \CJ \backslash \N, r \in \N \right\}.\label{model:2link}
	\end{align}
	The \LP relaxations of problems \eqref{2linkall} and \eqref{model:2link} are equivalent in terms of providing the same optimal value.
\end{theorem}
To prove \cref{prop:2link}, we first provide an optimality condition for the \LP relaxation of problem \eqref{model:2link}.
Given a feasible solution $(x^*, y^*)$ of the \LP relaxation of problem \eqref{model:2link}, let
\begin{align}
	& p_j = \argmax\{x^*_s  - y^*(\CI_s \backslash \CI_j) \,:\,  s \in \CJ\backslash\N \}~\text{for}~j \in \N \label{2linkpj},\\
	& n_j = \argmin\{x^*_s + y^*(\CI_j \backslash \CI_s) \,:\,  s \in \N \}~\text{for}~j \in \CJ\backslash\N.	\label{2linknj}
\end{align}
If $\CJ \backslash \N=\varnothing$, we let $p_j =0$, $\CI_{p_j}=\varnothing$, and $x^*_{p_j}=0$; and if $\N=\varnothing$, we let $n_j=-1$,  $\CI_{n_j}=\varnothing$, and $x^*_{n_j}=1$.
Then, it is easy to see that there exists an optimal solution $(x^*, y^*)$ of the {\rm \LP} relaxation of problem \eqref{model:2link} satisfying
\begin{equation}\label{opt3}
	x^*_j =\left\{ \begin{array}{ll} 
		\max\left\{\max_{i \in \CI_j} y^*_i, x^*_{p_j} - y^*(\CI_{p_j} \backslash \CI_{j}) \right\}, & \text{if}~j \in \N; \\
		\min\{1, y^*(\CI_j), x^*_{n_j} +  y^*(\CI_j \backslash \CI_{n_j}) \},&\text{otherwise},
	\end{array} \right.
	\forall~j \in \CJ.
\end{equation}

\begin{proof}[Proof of Theorem 5.2]
	Let $o_1$ and $o_2$ be the optimal values of the \LP relaxations of problems \eqref{2linkall} and \eqref{model:2link}, respectively. 
	Clearly, $o_1 \leq o_2$ holds. 
	To show $o_1 \geq o_2$, it suffices to show that for an optimal solution $(x^*, y^*)$ of the \LP relaxation of \eqref{model:2link} satisfying \eqref{opt3}, it follows 
	$x^*_j \leq x^*_r + y^*(\CI_j \backslash \CI_r)$ for all (i) $j, r \in \CJ \backslash \N$, (ii) $j, r \in \N$, and (iii) $j \in \N, r \in \CJ \backslash \N$.
	Observe that by \eqref{opt3}, if $\N = \varnothing$, then $x_j^* = \min \{ 1, y^*(\CI_j) \} \leq \min\{ 1+ y^*(\CI_j \backslash \CI_r), y^*(\CI_r) + y^*(\CI_j \backslash \CI_r) \} = \min\{ 1, y^*(\CI_r)  \}+ y^*(\CI_j \backslash \CI_r) = x_r^* +  y^*(\CI_j \backslash \CI_r)$ where $j, r \in \CJ$; 
	if $\CJ \backslash \N = \varnothing$, then  
	$x_j^* = \max_{i \in \CI_j}y_i^*\leq \max_{i \in \CI_r}y_i^*+ \max_{i \in \CI_j \backslash \CI_r}y_i^* \leq \max_{i \in \CI_r}y_i^*+y^*(\CI_j \backslash \CI_r) =  x^*_r +  y^*(\CI_j \backslash \CI_r)$ where $j,r \in \CJ$.
	In both case, $x^*_j \leq x^*_r + y^*(\CI_j \backslash \CI_r)$ holds for all $j,r \in \CJ$. 
	Thus, we can assume $\N \neq \varnothing$ and $\CJ \backslash \N \neq \varnothing$.
	We consider the three cases (i)--(iii), separately.
	\begin{itemize}
		\item[(i)] $j, r \in \CJ \backslash \N$.
		From the definition of $n_j$ in \eqref{2linknj} and $n_r \in \N$, we have $x^*_{n_j} + y^*(\CI_j \backslash \CI_{n_j}) \leq x^*_{n_r} + y^*(\CI_j \backslash \CI_{n_r}) \leq x^*_{n_r} + y^*(\CI_j \backslash \CI_r) + y^*(\CI_r \backslash \CI_{n_r})$.
		Together with $y^*(\CI_r) + y^*(\CI_j \backslash \CI_r) \geq y^*(\CI_j)$, we obtain
		\begin{equation*}
			\begin{aligned}
				x^*_j &= \min\{1, \ y^*(\CI_j), \ x^*_{n_j} + y^*(\CI_j \backslash \CI_{n_j}) \}  \\
				&\leq \min \{1+y^*(\CI_j \backslash \CI_r), \ y^*(\CI_r) + y^*(\CI_j \backslash \CI_r), \ x^*_{n_r} + y^*(\CI_j \backslash \CI_r) + y^*(\CI_r \backslash \CI_{n_r}) \} \\
				& = \min \{1, \ y^*(\CI_r), \ x^*_{n_r} + y^*(\CI_r \backslash \CI_{n_r}) \} + y^*(\CI_j \backslash \CI_r) = x^*_r + y^*(\CI_j \backslash \CI_r).
			\end{aligned}
		\end{equation*}
		\item[(ii)] $j, r \in \N$. 
		From the definition of $p_r $ in \eqref{2linkpj} and  $p_j \in \CJ\backslash\N$, we have  $x^*_{p_r} - y^*(\CI_{p_r} \backslash \CI_{r}) \geq x^*_{p_j} - y^*(\CI_{p_j} \backslash \CI_{r}) \geq x^*_{p_j} - y^*(\CI_{p_j} \backslash \CI_{j}) - y^*(\CI_j \backslash \CI_r)$,
		and thus $x^*_{p_r} - y^*(\CI_{p_r} \backslash \CI_{r}) + y^*(\CI_j \backslash \CI_r)\geq x^*_{p_j} - y^*(\CI_{p_j} \backslash \CI_{j})$.
		Together with $\max_{i \in \CI_j} y^*_i \leq \max_{i \in \CI_r} y^*_i + \max_{i \in \CI_j \backslash \CI_r} y^*_i \leq \max_{i \in \CI_r} y^*_i + y^*(\CI_j \backslash \CI_r)$, we obtain
		\begin{equation*}
			\begin{aligned}
				x^*_j  
				&= \max\left\{\max_{i \in \CI_j} y^*_i, \ x^*_{p_j} - y^*(\CI_{p_j} \backslash \CI_{j}) \right\} \\
				&\leq \max\left\{\max_{i \in \CI_r} y^*_i + y^*(\CI_j \backslash \CI_r), \ x^*_{p_r} - y^*(\CI_{p_r} \backslash \CI_{r}) + y^*(\CI_j \backslash \CI_r) \right\} \\
				&= \max\left\{\max_{i \in \CI_r} y^*_i, \ x^*_{p_r} - y^*(\CI_{p_r} \backslash \CI_{r}) \right\} + y^*(\CI_j \backslash \CI_r) = x^*_r + y^*(\CI_j \backslash \CI_r).
			\end{aligned}
		\end{equation*}
		\item[(iii)] $j \in \N, r \in \CJ \backslash \N$.
		As $n_r\in \N$, we have $x^*_{n_r} \geq \max_{i \in \CI_{n_r}} y^*_i$, which, together with $ y^*(\CI_j \backslash \CI_r) + y^*(\CI_{r} \backslash \CI_{n_r}) \geq  y^*(\CI_{j} \backslash \CI_{n_r}) $ and $\max_{i \in \CI_{n_r}} y^*_i +  y^*(\CI_{j} \backslash \CI_{n_r}) \geq \max_{i \in \CI_{n_r}} y^*_i +  \max_{i \in \CI_j\backslash  \CI_{n_r}}y^*_i \geq \max_{i \in \CI_j} y^*_i$, implies 	 
		(a) $x^*_{n_r} + y^*(\CI_j \backslash \CI_r) + y^*(\CI_{r} \backslash \CI_{n_r})\geq \max_{i \in \CI_j} y^*_i$.
		From $p_j \in \CJ\backslash \N$ and $n_r \in \N$, we obtain (b) $x^*_{p_j} \leq x^*_{n_r} + y^*(\CI_{p_j} \backslash \CI_{n_r})\leq x^*_{n_r} +  y^*(\CI_{p_j} \backslash \CI_{j}) + y^*(\CI_{j} \backslash \CI_{r}) + y^*(\CI_{r} \backslash \CI_{n_r})$, or equivalently,  $x^*_{n_r} +  y^*(\CI_{j} \backslash \CI_{r})  + y^*(\CI_{r} \backslash \CI_{n_r}) \geq x^*_{p_j}-y^*(\CI_{p_j} \backslash \CI_{j}) $.
		Combining (a) and (b) yields
		\begin{equation}\label{geq1}
			x^*_{n_r} + y^*(\CI_{j} \backslash \CI_{r}) + y^*(\CI_{r} \backslash \CI_{n_r}) \geq \max \left\{\max_{i \in \CI_j} y^*_i, \ x^*_{p_j} - y^*(\CI_{p_j} \backslash \CI_{j}) \right\} .	
		\end{equation}
		From $p_j \in \CJ\backslash \N$, we have $x_{p_j}^* \leq y^*(\CI_{p_j}) \leq y^*(\CI_{p_j} \backslash \CI_{j}) + y^*(\CI_{j} \backslash \CI_{r}) + y^*(\CI_{r}) $. 
		This, together with  $y^*(\CI_{j} \backslash \CI_{r}) + y^*(\CI_r) \geq y^*(\CI_j)\geq  \max_{i \in \CI_j} y^*_i$, indicates
		\begin{equation}\label{geq2}
			y^*(\CI_{j} \backslash \CI_{r}) + y^*(\CI_{r}) \geq \max \left\{\max_{i \in \CI_j} y^*_i, \ x^*_{p_j} - y^*(\CI_{p_j} \backslash \CI_{j})\right\} .
		\end{equation}
		Combining \eqref{geq1}, \eqref{geq2}, and $x^*_j =\max\{\max_{i \in \CI_j} y^*_i, \ x^*_{p_j} - y^*(\CI_{p_j} \backslash \CI_{j})\} \leq 1 \leq 1 + y^*(\CI_{j} \backslash \CI_{r})$, we obtain
		\begin{equation*}
			\begin{aligned}
				x^*_j &= \max\left\{\max_{i \in \CI_j} y^*_i, \ x^*_{p_j} - y^*(\CI_{p_j} \backslash \CI_{j}) \right\} \\
				&\leq \min \left\{1 + y^*(\CI_{j} \backslash \CI_{r}),~y^*(\CI_{j} \backslash \CI_{r}) + y^*(\CI_r),~x^*_{n_r} + y^*(\CI_{j} \backslash \CI_{r}) + y^*(\CI_{r} \backslash \CI_{n_r}) \right\} \\
				&= \min \left\{1,~y^*(\CI_r),~x^*_{n_r} + y^*(\CI_{r} \backslash \CI_{n_r}) \right\} + y^*(\CI_{j} \backslash \CI_{r})= x^*_r + y^*(\CI_{j} \backslash \CI_{r}).
			\end{aligned}
		\end{equation*}
	\end{itemize}
	\vspace{-1cm}
\end{proof}

Note that for the \MCLP or \MinCLP (i.e., the special case of the \MCLPNW with $\N=\varnothing$ or $\CJ \backslash \N=\varnothing$, respectively), no two-customer inequality is included in problem \eqref{model:2link}. 
Therefore, the equivalence of the \LP relaxations of problems \eqref{2linkall} and \eqref{model:2link} in \cref{prop:2link} implies that the two-customer inequalities \eqref{valid1} cannot improve the \LP relaxation bound of the \MIP formulation of the \MCLP or \MinCLP.
\begin{corollary}\label{coro11}
		For the classic \MCLP or \MinCLP, the LP relaxation of the \MIP formulation with the two-customer inequalities is equivalent to that of the \MIP formulation without the two-customer inequalities.
\end{corollary}

The following proposition further provides a necessary condition for the two-customer inequality $x_j \leq x_r + y(\CI_j \backslash \CI_r)$ with $j \in \CJ\backslash\N $ and $r \in \N $ to strengthen the \LP relaxation of the \MCLPNW.

\begin{proposition}\label{re2:validineq}
	Let $j \in \CJ\backslash\N$ and $r \in \N$. If $|\CI_j \cap \CI_r| \leq 1$,  inequality \eqref{valid1} is dominated by other inequalities in formulation \eqref{model:2link}.
\end{proposition}
\begin{proof}
	If $|\CI_j \cap \CI_r| = 0$, then inequality \eqref{valid1} reduces to $x_j \leq x_r +  y(\CI_j) $ and thus is dominated by inequality $x_j \leq y(\CI_j)$. 
	Otherwise, $\CI_j \cap \CI_r = \{i'\}$ holds for some $i' \in \CI$.
	In this case, inequality \eqref{valid1} reduces to $x_j \leq x_r +  y(\CI_j\backslash\{i'\}) $ and is dominated by inequalities $x_j \leq  y(\CI_j) $ and $y_{i'} \leq x_r$.
\end{proof}

Combining \cref{prop:2link} and \cref{re2:validineq}, we can conclude that
in order to use the two-customer inequalities \eqref{valid1} to strengthen the \LP relaxation of formulation \eqref{mclp-nw1}, it suffices to consider those with $j \in \CJ\backslash 
\N$, $r \in \N$, and $|\CI_j \cap \CI_r|\geq 2$.

\begin{example}\label{ex:twolink}
	Consider an example of the \MCLPNW where $p = 1$ and there exist three customers and four candidate facility locations.
	The weights of the three customers are $w_1 = 1$, $w_2 = -1$, and $w_3 = -1$, and
	$\CI_1 = \{2, 3, 4\}$, $\CI_2 = \{1, 2, 3\}$, and $\CI_3 =\{1,4\}$.
	In this example, no dominance inequality exists and the  \LP relaxation of formulation \eqref{mclp-nw1} reads
	\begin{equation}\label{temp2}
		\begin{aligned}
			z_{\rm LP} =  \max_{(x, y)\in [0, 1]^3\times [0, 1]^4} & \left\{ x_1-x_2 -x_3 \,:\, y_1+y_2+y_3+y_4 = 1,~y_2+y_3+y_4 \geq x_1,\right. \\
			& \qquad\qquad\qquad \left. x_2 \geq y_1,~x_2 \geq y_2,~x_2 \geq y_3,~x_3 \geq y_1,~x_3 \geq y_4\right\},
		\end{aligned}
	\end{equation}
	where an optimal solution of problem \eqref{temp2} is given by $(\hat{x}, \hat{y}) = (1, \frac{1}{2}, 0, 0, \frac{1}{2}, \frac{1}{2}, 0)$ with an objective value of $\frac{1}{2}$.
	From \cref{prop:2link} and \cref{re2:validineq}, among the six two-customer inequalities, only $x_1 \leq x_2 +y_4$ could strengthen the \LP relaxation \eqref{temp2}. Moreover, it can cut off the optimal solution $(\hat{x}, \hat{y})$ of the \LP relaxation \eqref{temp2}.
	Adding it  into the problem, we obtain
	\begin{equation*}
		\begin{aligned}
			z'_{\rm LP} =  \max_{(x, y)\in [0, 1]^3\times [0, 1]^4} & \left\{ x_1-x_2 -x_3 \,:\, y_1+y_2+y_3+y_4 = 1,~y_2+y_3+y_4 \geq x_1,\right. \\
			& \qquad \left. x_2 \geq y_1,~x_2 \geq y_2,~x_2 \geq y_3,~x_3 \geq y_1,~x_3 \geq y_4,~x_1 \leq x_2 + y_4\right\}.
		\end{aligned}
	\end{equation*}
	By simple computation, we can check that $(x^*,y^*) = (1, 0, 1, 0, 0, 0, 1)$ is an optimal solution of the above problem.
	Therefore, $z'_{\rm LP}=0 < z_{\rm LP}= \frac{1}{2}$.
\end{example}

\subsection{Separation}

Observe that due to the potentially huge number of the two-customer inequalities \eqref{valid1} (with $j \in \CJ\backslash \N$, $r \in \N$, and $|\CI_j \cap \CI_r|\geq 2$), directly adding them into formulation \eqref{mclp-nw1} may lead to a large \LP relaxation, making the resultant problem inefficient to be solved by \MIP solvers.
Therefore, we use a branch-and-cut approach in which inequalities \eqref{valid1} are separated on the fly.
Specifically, we first compute $\C= \{ (j,r)\,: \, j \in \CJ\backslash \N, ~r \in \N,~ |\CI_j \cap \CI_r|\geq 2 \}$.
Then for the current \LP relaxation solution $(\bar{x}, \bar{y})$ encountered during the branch-and-cut approach, we add, for each $(j,r) \in \C$, $x_j \leq x_r + y(\CI_j\backslash \CI_r)$ into the problem if it is violated by $(\bar{x}, \bar{y})$. 
Overall, the complexity of the separation algorithm is upper bounded by $\mathcal{O}(|\CJ| \sum_{j \in \CJ} |\CI_j|)$.

\section{Computational results}\label{sect:num}
In this section, we present computational results to demonstrate the effectiveness of the proposed isomorphic aggregation, dominance reduction, and two-customer inequalities for solving the \MCLPNW.
To do this, we first perform numerical experiments to demonstrate the effectiveness of embedding the three proposed techniques into a branch-and-cut solver.
\rev{Then, we compare our approach (i.e., using an \MIP solver with the three proposed techniques) with an extension of the state-of-the-art \BD in \cite{Cordeau2019}.}
Finally, we present computational results to evaluate the effect of using each technique for solving the \MCLPNW\footnote{In Appendix D of the online supplement, we also present computational results to demonstrate the effectiveness of the proposed  isomorphic aggregation, dominance reduction, and two-customer inequalities for solving a variant of the \MCLPNW that additionally takes the distance constraints of the facilities \citep{Moon1984,Berman2008,Grubesic2012} into account.}.

The proposed isomorphic aggregation, dominance reduction, and two-customer inequalities were implemented in Julia 1.7.3 using \CPLEX 20.1.0.
The parameters of \CPLEX were configured to run the code in a single-threaded mode, 
with a time limit of 7200 seconds and a relative \MIP gap tolerance of $0\%$.
Unless otherwise stated, all other parameters in \CPLEX were set to their default values. 
All computational experiments were performed on a cluster of Intel(R) Xeon(R) Gold 6140 CPU @ 2.30GHz computers.

We use two testsets of instances, namely, \DATAONE and \DATATWO.
Testset \DATAONE contains $240$ \MCLPNW instances with identical numbers of candidate facility locations and customers. $40$ instances of them were constructed by \cite{Berman2009} using the $p$-median instances from OR-Library \citep{Beasley1990}, and have up to $900$ candidate facility locations and customers and $p$ values ranging between $5$ and $200$.
According to \cite{Berman2009}, the coverage distance $R$ is computed as the $\frac{1}{2p}$ percentile of the distances between all pairs of customers,
and odd- and even-numbered customers are given a weight of $+1$ and $-1$, respectively.
In addition to this setting, we also construct $200$ \MCLPNW instances with non-unit weights and varying ratios $r$ of customers 
with negative weights (or positive weights), as to better reflect real-world scenarios. 
Specifically, for each $r \in \{0.1, 0.3, 0.5, 0.7, 0.9\}$, 
we randomly allocate negative weights to $r\times|\CJ|$  customers, while the remaining customers are assigned positive weights.
The positive and negative customer weights are uniformly selected from $\{1, 2, \ldots, 100\}$ and $\{-100, -99, \ldots, -1\}$, respectively.

Testset \DATATWO consists of $336$ \MCLPNW instances whose number of customers is much larger than the number of candidate facility locations.
We use a similar procedure as in \cite{Cordeau2019} to construct the instances in testset \DATATWO.
The numbers of customers $|\CJ|$ and candidate facility locations $|\CI|$ are chosen from $\{1000, 10000\}$ and $\{100, 200\}$, respectively.
The locations of all customers and candidate facilities are randomly chosen within a $30 \times 30$ region on the plane
and the distance $d_{ij}$ between candidate facility location $i$ and customer $j$ is calculated using the Euclidean distance metric.
The choices of the number of open facilities $p$ and the coverage distance $R$ are described in \cref{table_p_R}.
Similar to testset \DATAONE, testset \DATATWO is also further divided into six groups: \UNIT and \NU-$r$, $r \in \{0.1, 0.3, 0.5, 0.7, 0.9\}$, 
	where $\UNIT$ and  \NU-$r$ consist of $56$ instances with unit weights of $+1$ and $-1$, and non-unit weights, respectively. 
	In particular, for instances in group $\UNIT$, we assign a weight of $+1$ to the odd-numbered customers and $-1$ to the even-numbered customers;
	for instances in groups \NU-$r$, 
	we randomly designate $r\times|\CJ|$ customers with negative weights,  uniformly selected from $\{-100, -99, \ldots, -1\}$,  and the remaining customers are assigned positive weights, uniformly chosen from $\{1, 2, \ldots, 100\}$.

\begin{table}[!htbp]
	\centering
	\setlength{\tabcolsep}{10pt}
	\renewcommand{\arraystretch}{1.1}
	\caption{Parameters of the instances in testset \DATATWO.}
	\label{table_p_R}
	\begin{tabular}{ll}  
		\hline    
		\multirow{1}{*}{$p$} & \multicolumn{1}{l}{$R$}  \\  
		\cline{1-2}  
		$ 10\%|\CI|$  & \multicolumn{1}{l}{$R \in  \{5.5, 5.75, 6, 6.25\}$}  \\   
		$15\%|\CI|$  & \multicolumn{1}{l}{$R\in \{4, 4.25, 4.5, 4.75, 5\}$} \\   
		$ 20\%|\CI|$ & \multicolumn{1}{l}{$R\in \{3.25, 3.5, 3.75, 4, 4.25\}$}  \\  
		\hline      
	\end{tabular}
\end{table}

\subsection{Effectiveness of the three proposed techniques}\label{sec:effIDT}

We first present computational results to show the effectiveness of embedding the proposed isomorphic aggregation, dominance reduction, and two-customer inequalities into the branch-and-cut solver \CPLEX for solving the \MCLPNW.
In particular, we compare the following \rev{three} settings:
\begin{itemize}
	\item \CPX: formulation \eqref{mclp-nw1} is solved using \CPLEX's branch-and-cut algorithm;
	\item \rev{\CPXC: formulation \eqref{mclp-nw1} is solved using \CPX with the presolving techniques P1--P4 of \cite{Chen2023b};}
	\item \IDI: formulation \eqref{mclp-nw1} is solved using \CPXC with the proposed isomorphic aggregation, dominance reduction, and two-customer inequalities.
\end{itemize}

\cref{T1:average,T2:average} summarize the computational results of settings \CPX, \CPXC, and \IDI on the \rev{instances} in testsets \DATAONE and \DATATWO, respectively.
Detailed statistics of instance-wise computational results can be found in Tables F.4--F.15  of the online supplement.
The two tables present results for $40$ and $56$ instances per row, respectively, grouped by unit weights and non-unit weights with different ratios of negative customer weights.
Each row reports the average\footnote{Throughout this paper, all averages are taken to be geometric means with a shift of $1$ (the shifted geometric mean of values $x_1, x_2, \ldots, x_n$ with shift $s$ is defined as $\prod_{k=1}^n{(x_k + s)}^{1/n}-s$; see \cite{Achterberg2007}).} percentage of \LP gap (\LPG) computed as  $\frac{z_{\LP}- z}{z} \times 100\%$,
where $z$ is the objective value of the optimal solution or the best incumbent of the \MCLPNW and $z_{\LP}$ is the optimal value of its \LP relaxation.
Under each setting, we report the number of solved instances (\tblS),
the average (total) CPU time in seconds (\tblT), the average number of explored nodes (\tblN), and the average percentage of \emph{gap improvement} defined by
 \begin{equation*}
 	\tblGPC =  \frac{z_{\LP}-z_{\text{root}} }{z_{\LP}-z} \times 100 \%.
 \end{equation*}
Here, $z_{\text{root}}$ is the \LP relaxation bound obtained at the root node.
Under settings \CPXC and \IDI, we additionally report the average percentage reduction in the number of variables (\tblFV) and constraints (\tblDC), and the average CPU time spent in the implementation of the presolving techniques in seconds (\tblPT).
Under setting \IDI, we report the average CPU time spent in the separation of the two-customer inequalities in seconds (\tblST).
To intuitively compare the performance of \CPX, \CPXC, and \IDI, we plot the performance profiles of the (total) CPU time and number of explored nodes in \cref{fig:cpxvsall}.

First, as shown in \cref{T1:average}, for instances in $\NU$-$0.1$ (i.e., the number of customers with negative weights is much smaller than that of customers with positive weights) of testset \DATAONE, the gap between the optimal value of  formulation (1) and its \LP relaxation is small, 
and thus all these instances can be efficiently solved by all three settings: \CPX, \CPXC, and \IDI. 
In contrast, for instances with a fairly large number of customers with negative weights (i.e., instances in  \UNIT and \NU-$r$, $r \in \{0.3,0.5,0.7,0.9\}$), the \LP relaxation of formulation (1) is usually very weak, and thus it is more difficult to solve these instances by \CPX.
Second, we can observe from \cref{T1:average} that for instances in testset \DATAONE, the reductions by the presolving techniques P1--P4 of \cite{Chen2023b} are not large, and thus we do not observe a relatively large performance improvement of \CPXC over \CPX.  
In contrast, the three proposed techniques enable us to reduce the problem size and substantially strengthen the \LP relaxation of formulation \eqref{mclp-nw1}.
In particular, the three proposed techniques enable us to remove $2.3\%$--$5.0\%$ variables and $5.6\%$--$20.4\%$ constraints from the problem formulation, and achieve a much better gap improvement than \CPX and \CPXC. 
Due to the smaller problem size and, particularly, the much tighter \LP relaxation, the performance of \IDI is much better than that of \CPX and \CPXC, especially for the relatively difficult instances in \UNIT and \NU-$0.5$ (where the number of customers with negative weights is identical to that of customers with positive weights).
Overall, \IDI can solve $221$ instances among the $240$ instances to optimality while \CPX and \CPXC can only solve $196$ of them to optimality; \IDI generally enables us to return a much smaller CPU time and number of explored nodes than those returned by \CPX and \CPXC.  
The latter is further confirmed by \cref{T1:time,T1:node}, where the red-triangle line corresponding to \IDI is generally higher than the blue-circle and black-star lines corresponding to \CPX and \CPXC, respectively.
It is worthwhile remarking that for instances in \UNIT of testset \DATAONE, only $21$ instances were solved to optimality by \cite{Berman2009} while $34$ instances can be solved to optimality by the proposed \IDI.
In \cref{T1:IDT}, we present the results of the $13$ newly solved instances.

\begin{table}[!h]
	\centering
	\footnotesize
	\addtolength{\tabcolsep}{-2pt}
	\renewcommand{\arraystretch}{1}
	\caption{Performance comparison of settings \CPX, \CPXC, and \IDI on the instances in testset \DATAONE.} \label{T1:average}
	\begin{tabular}{c*{20}{r}}
		\toprule
		\multirow{2}{*}{Groups} & \multirow{2}{*}{\LPG} & \multicolumn{4}{c}{\CPX} & \multicolumn{7}{c}{\CPXC} & \multicolumn{8}{c}{\IDI} \\
		\cmidrule(r{0.1em}){3-6}  \cmidrule(lr{1em}{0.1em}){7-13}  \cmidrule(lr{1em}{0.3em}){14-21} 
		&& \tblS & \multicolumn{1}{c}{\tblT}  & \multicolumn{1}{r}{\tblN} & \multicolumn{1}{r}{\tblGPC} 
		& \tblS  & \multicolumn{1}{c}{\tblT}   & \multicolumn{1}{r}{\tblN} & \multicolumn{1}{r}{\tblGPC}  & \multicolumn{1}{r}{\tblFV} & \multicolumn{1}{r}{\tblDC} & \multicolumn{1}{r}{\tblPT}
		& \tblS  & \multicolumn{1}{c}{\tblT}   & \multicolumn{1}{r}{\tblN} & \multicolumn{1}{r}{\tblGPC}  & \multicolumn{1}{r}{\tblFV} & \multicolumn{1}{r}{\tblDC}   & \multicolumn{1}{r}{\tblPT} & \multicolumn{1}{r}{\tblST} \\
		\hline 
		\UNIT &    57.6 &               28 & 55.2 & 232 & 39.0 &            28 & 60.7 & 199 & 40.0 & 2.5 & 1.7 & 2.5 &          34 & 19.7 & 4 & 93.8 & 4.4 & 13.8 & 1.3 & 2.6 \\ 
		\NU-0.1 &    3.2 &               40 & 2.9 & 20 & 52.4 &            40 & 4.0 & 19 & 54.1 & 4.7 & 5.4 & 0.5 &          40 & 3.2 & 4 & 92.2 & 5.0 & 5.6 & 0.8 & 0.9 \\ 
		\NU-0.3 &    13.6 &               33 & 21.0 & 193 & 41.3 &            34 & 20.1 & 168 & 42.6 & 3.7 & 3.0 & 1.3 &          37 & 5.5 & 5 & 94.3 & 4.3 & 9.6 & 0.3 & 1.4 \\ 
		\NU-0.5 &    37.1 &               30 & 39.0 & 164 & 36.7 &            29 & 37.1 & 119 & 38.3 & 2.6 & 1.8 & 1.6 &          34 & 10.4 & 4 & 93.5 & 3.7 & 12.9 & 0.4 & 1.8 \\ 
		\NU-0.7 &    78.2 &               30 & 36.6 & 152 & 38.2 &            31 & 37.0 & 160 & 39.3 & 1.6 & 0.9 & 1.3 &          36 & 7.3 & 3 & 95.1 & 3.0 & 16.8 & 0.2 & 1.8 \\ 
		\NU-0.9 &    178.6 &               35 & 13.7 & 40 & 48.3 &            34 & 13.6 & 40 & 48.3 & 0.5 & 0.3 & 0.6 &          40 & 3.3 & 2 & 96.7 & 2.3 & 20.4 & $<$0.1 & 1.1 \\ 
		\bottomrule
	\end{tabular}
\end{table}

\begin{table}[!h]
	\centering
	\footnotesize
	\addtolength{\tabcolsep}{-3.2pt}
	\renewcommand{\arraystretch}{1}
	\caption{Performance comparison of settings \CPX, \CPXC, and \IDI on the instances in testset \DATATWO.} \label{T2:average}
	\begin{tabular}{c*{20}{r}}
		\toprule
		\multirow{2}{*}{Groups} & \multirow{2}{*}{\LPG} & \multicolumn{4}{c}{\CPX} & \multicolumn{7}{c}{\CPXC} & \multicolumn{8}{c}{\IDI} \\
		\cmidrule(r{0.1em}){3-6}  \cmidrule(lr{1em}{0.1em}){7-13}  \cmidrule(lr{1em}{0.3em}){14-21} 
		&& \tblS & \multicolumn{1}{c}{\tblT}  & \multicolumn{1}{r}{\tblN} & \multicolumn{1}{r}{\tblGPC} 
		& \tblS  & \multicolumn{1}{c}{\tblT}   & \multicolumn{1}{r}{\tblN} & \multicolumn{1}{r}{\tblGPC}  & \multicolumn{1}{r}{\tblFV} & \multicolumn{1}{r}{\tblDC} & \multicolumn{1}{r}{\tblPT}
		& \tblS  & \multicolumn{1}{c}{\tblT}   & \multicolumn{1}{r}{\tblN} & \multicolumn{1}{r}{\tblGPC}  & \multicolumn{1}{r}{\tblFV} & \multicolumn{1}{r}{\tblDC}   & \multicolumn{1}{r}{\tblPT} & \multicolumn{1}{r}{\tblST} \\
		\hline 
		\UNIT &     685.6 &           4 & 5348.5 & 42842 & 39.6 &            14 & 3499.0 & 35798 & 55.6 & 26.4 & 4.5 & 18.5 &          56 & 18.0 & 137 & 98.7 & 70.7 & 82.7 & 2.3 & 3.5 \\ 
		\NU-0.1 &     5.4 &        56 & 7.6 & 17 & 87.3 &            56 & 2.8 & 1 & 98.9 & 57.6 & 29.3 & 1.3 &          56 & 2.8 & $<$1 & 99.9 & 65.8 & 71.4 & 1.8 & 0.2 \\ 
		\NU-0.3 &     43.6 &        28 & 387.0 & 7945 & 52.9 &            41 & 137.2 & 3622 & 74.5 & 41.5 & 10.4 & 4.0 &          56 & 2.1 & 2 & 99.7 & 65.7 & 75.9 & 1.0 & 0.5 \\ 
		\NU-0.5 &     705.0 &        5 & 5540.7 & 42864 & 37.9 &            12 & 3529.9 & 35733 & 54.4 & 26.6 & 4.6 & 16.9 &          56 & 24.3 & 336 & 97.8 & 65.6 & 79.1 & 1.4 & 3.7 \\ 
		\NU-0.7 &     277.8 &        15 & 2687.1 & 7628 & 44.4 &            19 & 2281.6 & 7145 & 56.9 & 13.2 & 1.8 & 5.2 &          56 & 16.7 & 463 & 86.9 & 65.4 & 81.8 & 1.2 & 1.2 \\ 
		\NU-0.9 &     67.5 &        41 & 409.1 & 1275 & 39.8 &            44 & 313.7 & 948 & 43.9 & 2.8 & 0.4 & 1.5 &          56 & 12.4 & 611 & 52.1 & 65.3 & 83.9 & 1.2 & 0.2 \\ 
		\bottomrule
	\end{tabular}
\end{table}

\begin{table}[!h]
	\centering
	\footnotesize
	\addtolength{\tabcolsep}{0pt}
	\renewcommand{\arraystretch}{1}
	\caption{Previously unsolved \MCLPNW instances in \citet{Berman2009} solved to optimality by the proposed \IDI.} \label{T1:IDT}
	\begin{tabular}{cc*{11}{r}}
		\toprule
		$|\CI|$ & $|\CJ|$  & $p$ &$R$  & $z_{\LP}$ & \tblLB  & \tblT & \tblN & \tblGPC & \tblFV & \tblDC & \tblPT & \tblST \\
		\hline 
		300 & 300 & 5 & 30 & 87.4 & 31 & 8.3 & 23 & 94.4 & 0.3 & 12.7 & 0.9 & 3.0 \\ 
		300 & 300 & 10 & 27 & 88.8 & 43 & 5.3 & 13 & 95.5 & 1.2 & 9.6 & 1.0 & 1.5 \\ 
		400 & 400 & 5 & 25 & 135.6 & 35 & 210.6 & 2469 & 87.7 & 0.9 & 9.6 & 1.2 & 7.0 \\ 
		400 & 400 & 10 & 21 & 120.2 & 58 & 10.9 & 82 & 92.9 & 0.9 & 10.3 & 1.0 & 2.3 \\ 
		400 & 400 & 40 & 14 & 118.4 & 90 & 2.5 & 0 & 100.0 & 6.0 & 14.4 & 0.8 & 0.7 \\ 
		500 & 500 & 5 & 23 & 169.8 & 48 & 1635.2 & 19962 & 84.5 & 0.1 & 3.5 & 1.9 & 13.7 \\ 
		500 & 500 & 10 & 21 & 169.0 & 82 & 170.6 & 1367 & 88.9 & 0.5 & 2.1 & 1.1 & 9.0 \\ 
		600 & 600 & 10 & 16 & 183.5 & 72 & 1701.6 & 28868 & 82.9 & 0.8 & 6.2 & 2.5 & 10.2 \\ 
		600 & 600 & 60 & 9 & 179.1 & 132 & 2.3 & 0 & 100.0 & 6.2 & 13.6 & 0.8 & 0.6 \\ 
		700 & 700 & 10 & 16 & 234.1 & 92 & 4953.9 & 43275 & 76.7 & 0.4 & 1.8 & 3.4 & 20.0 \\ 
		700 & 700 & 70 & 8 & 208.2 & 161 & 2.0 & 0 & 100.0 & 5.9 & 14.0 & 0.8 & 0.5 \\ 
		800 & 800 & 80 & 8 & 253.6 & 187 & 2.5 & 0 & 100.0 & 3.9 & 13.5 & 0.8 & 0.7 \\ 
		900 & 900 & 90 & 7 & 293.9 & 230 & 3.1 & 0 & 100.0 & 4.7 & 12.3 & 0.8 & 0.8 \\ 
		\bottomrule 
	\end{tabular}
\end{table}

\begin{figure}[!h]
	\centering
	\subfloat[\DATAONE]{\label{T1:time}\includegraphics[width=0.47\textwidth]{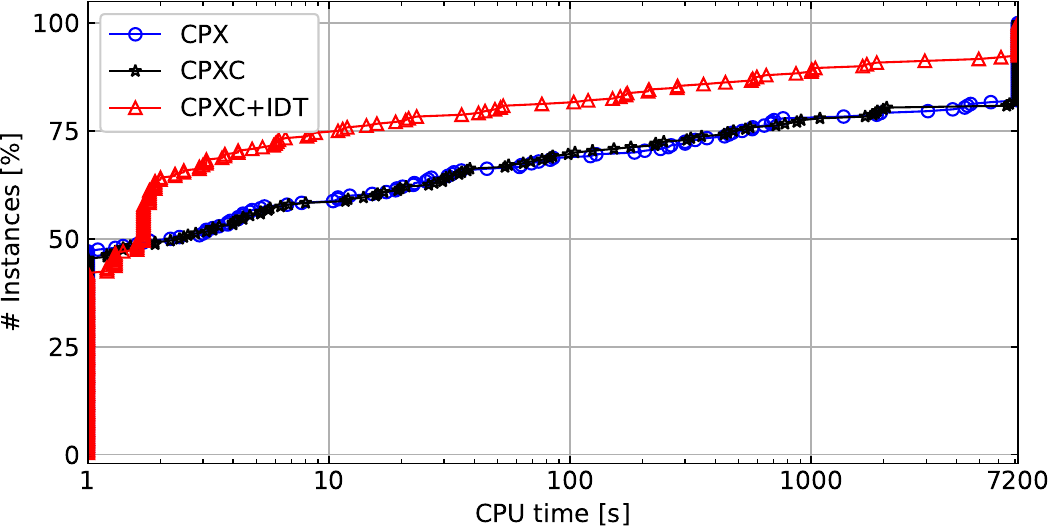}}
	\qquad
	\subfloat[\DATAONE]{\label{T1:node}\includegraphics[width=0.47\textwidth]{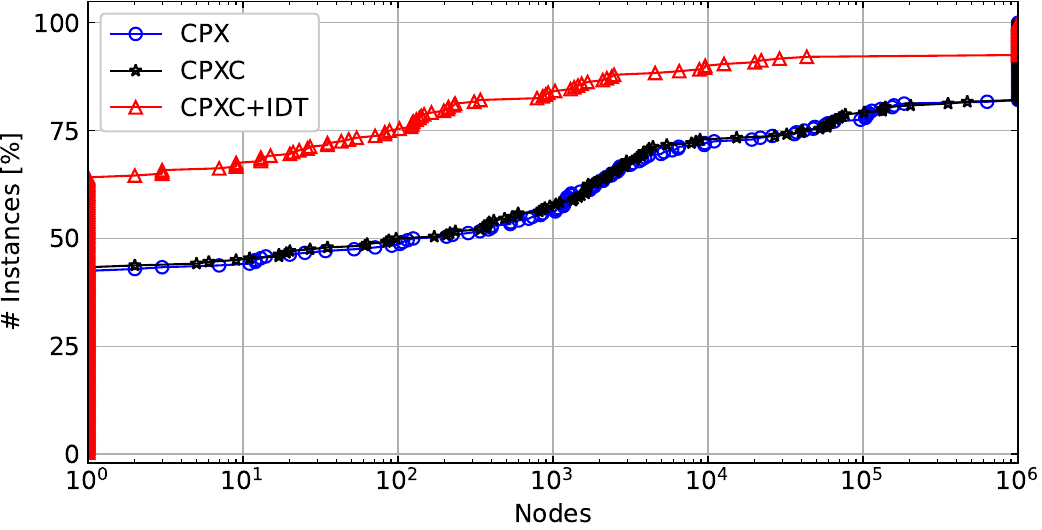}}
	\qquad
	\subfloat[\DATATWO]{\label{T2:time}\includegraphics[width=0.47\textwidth]{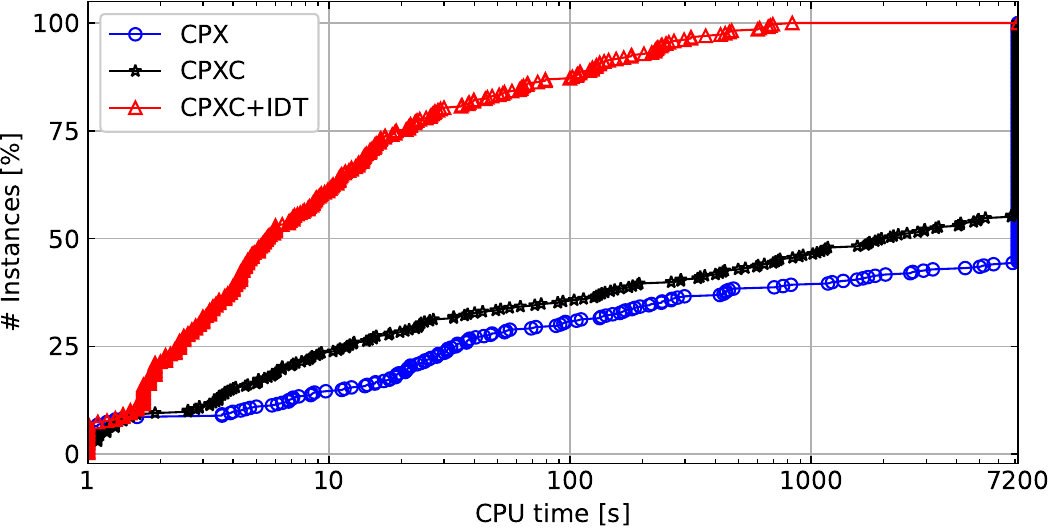}}
	\qquad
	\subfloat[\DATATWO]{\label{T2:node}\includegraphics[width=0.47\textwidth]{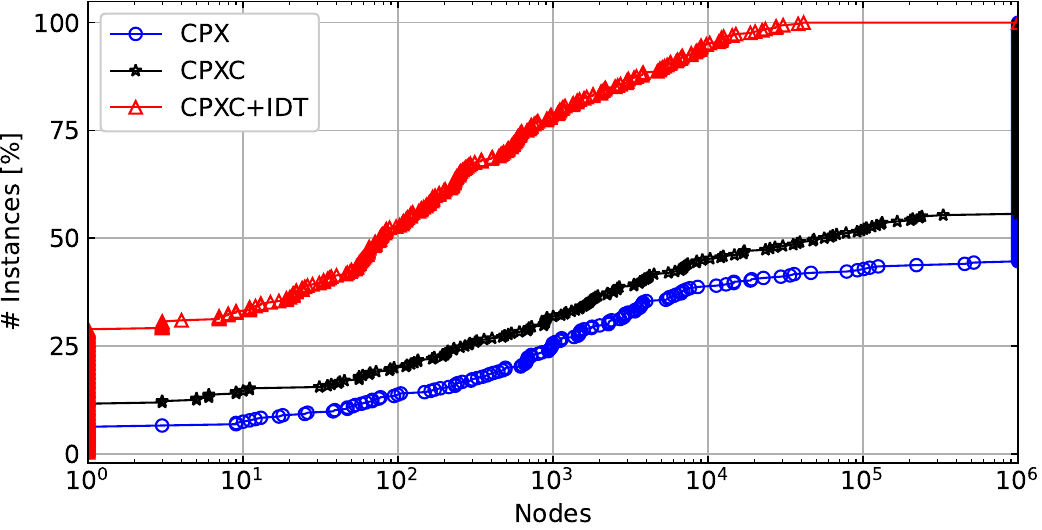}}
	\vspace*{-0.2cm}
	\caption{Performance profiles of the CPU time and number of explored nodes for settings \CPX, \CPXC, and \IDI.}
	\vspace*{-0.2cm}
	\label{fig:cpxvsall}
\end{figure}

For instances in testset \DATATWO, the performance improvement by the presolving techniques P1--P4 of \cite{Chen2023b} is relatively large but still not significant; see \cref{T2:time,T2:node}.
In contrast, we can observe a tremendous performance improvement by the three proposed techniques. 
In particular, with the three proposed techniques, we can observe a reduction of $65.3\%$--$70.7\%$ variables and $71.4\%$--$83.9\%$ constraints, and a gap improvement of $52.1\%$--$99.9\%$.
Overall, \IDI, equipped with the three proposed techniques, can solve all $336$ instances to optimality with an average solution time of $2.1$--$24.3$ seconds.
In sharp contrast, \CPX and \CPXC are only capable of solving $149$ and $186$ instances, respectively, to optimality within the time limit of 7200 seconds.
Indeed, for the relatively difficult instances in \UNIT and \NU-$0.5$ (i.e., instances with half of the customers with negative weights), \CPX  can only solve $4$ and $5$ instances to optimality, while \CPXC can only solve $14$ and $12$ instances to optimality.
These results highlight the efficiency of the three proposed techniques for solving realistic \MCLPNW{s} with a large number of customers, i.e., it can effectively turn them from intractable to easily solvable.

\subsection{Comparison with the state-of-the-art Benders decomposition}

In this subsection, we extend the state-of-the-art \BD of \cite{Cordeau2019} to solving the \MCLPNW, denoted as \BDC, and compare it with the proposed \IDI. 
A detailed discussion on the extension of the \BD to solving the \MCLPNW is provided in Appendix E of the online supplement.
In our implementation of the \BD, we apply the isomorphic aggregation to reduce the problem size of the \MCLPNW, as to accelerate the \BD. 
We do not apply the dominance reduction and two-customer inequalities as the Benders master problem does not contain variables $x$.

Detailed statistics of instance-wise computational results can be found in Tables F.16--F.27 of the online supplement.
\cref{fig:comparePreandBD} plots the performance profiles of the CPU times returned by \BDC and \IDI.
We can observe from \cref{fig:comparePreandBD} that \IDI significantly outperforms \BDC for instances in both testsets \DATAONE and \DATATWO. 
In particular, \IDI can solve $90\%$ of instances and all instances to optimality within the time limit of 7200 seconds in testsets \DATAONE and \DATATWO, respectively, 
while \BDC can only solve about $40\%$ and $30\%$ of instances to optimality in testsets \DATAONE and \DATATWO, respectively.
This is not surprising, since the efficiency of a \BD highly depends on the tightness of the \LP relaxation of the original formulation (or equivalently, the \LP relaxation of the Benders master problem;\citep{Rahmaniani2017}) and 
unfortunately, unlike the classic \MCLP whose \LP relaxation is usually tight or near tight \citep{ReVelle1993,Snyder2011,Cordeau2019}, the \MCLPNW suffers from an extremely weak \LP relaxation and thus the performance of the \BD is not competitive. 

\begin{figure}[!h]
	\centering 
	\subfloat[\DATAONE]{\label{T1:compchenT}\includegraphics[width=0.47\textwidth]{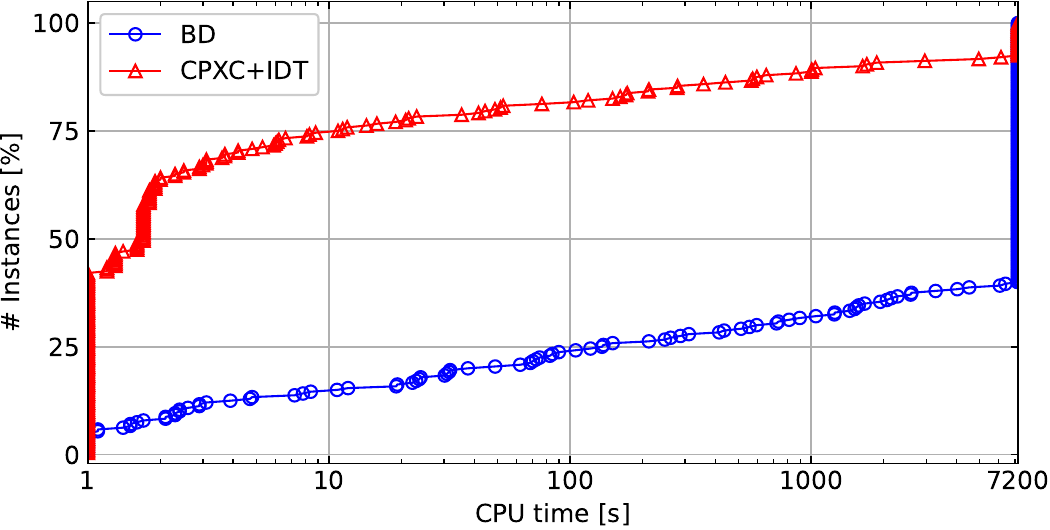}}
	\qquad
	\subfloat[\DATATWO]{\label{T1:compchenN}\includegraphics[width=0.47\textwidth]{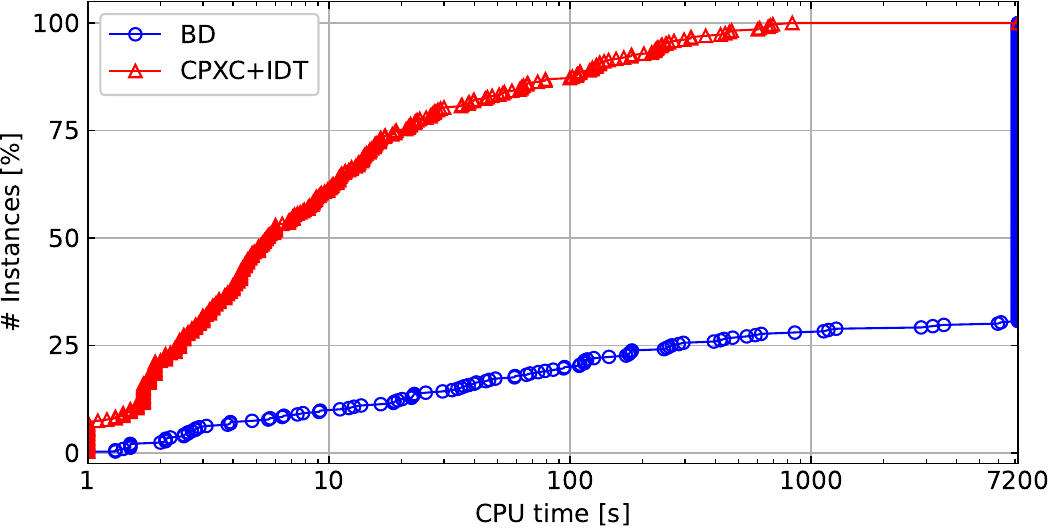}}
	\vspace*{-0.2cm}
	\caption{Performance profiles of the CPU time for settings \BDC and \IDI.}
	\vspace*{-0.2cm}
	\label{fig:comparePreandBD}
\end{figure}

\subsection{Effect of each technique}

Next, we evaluate the effect of using each technique for solving the \MCLPNW.
To do this, we compare the performance of \IDI with three settings, obtained by disabling one of the three proposed techniques of \IDI.
In the following, we use \NOAGG, \NODR, and \NOVI to denote \IDI with the isomorphic aggregation, dominance reduction, and two-customer inequalities disabled, respectively.

The performance comparison of \IDI with \NOAGG, \NODR, and \NOVI is summarized in \cref{eachsettings} and \cref{fig:eachtechnique}.
Detailed statistics of instance-wise computational results can be found in Tables F.28--F.39 of the online supplement.
In \cref{eachsettings}, columns \DeltaS and \RGPC denote the differences in the number of solved instances and the average percentage of gap improvement returned by each of the three settings (i.e., \NOAGG, \NODR, and \NOVI) and \IDI, respectively 
(a negative value under the three settings means that \IDI can solve more instances to optimality and return a better gap improvement).
Columns \RT and \RN display the ratios of the average CPU time and average number of explored nodes,
and columns \RV and \RC represent the average ratios of numbers of variables and constraints (a value greater than $1.0$ represents an improvement for \IDI).
We also plot the performance profiles of the CPU time and number of explored nodes  in \cref{fig:eachtechnique}.

For instances in testset \DATAONE, we observe from \cref{eachsettings} and \cref{T1:time1,T1:node1} that the two-customer inequalities have a fairly large positive impact.
In particular, we can observe an additional $46.32\%$ gap improvement of  \IDI over \NOVI, showing that  the two-customer inequalities can effectively strengthen the \LP relaxation of formulation \eqref{mclp-nw1}. 
With these inequalities, $24$ more instances can be solved to optimality, and the CPU time and number of explored nodes are reduced by factors of $3.04$ and $21.25$, respectively.
For the isomorphic aggregation or dominance reduction, the performance is, however, neutral, as illustrated in \cref{T1:time1,T1:node1}. 
This can be explained as follows.
First, the reductions on the number of variables and constraints by the two presolving techniques are relatively small (as shown in columns \RV and \RC of \cref{eachsettings}).
Second, the addition of the isomorphic aggregation (respectively, the dominance reduction) does not make a better gap improvement of \IDI over \NOAGG (respectively, over \NODR), which is due to the inclusion of the dominance reduction in \NOAGG(respectively, the two-customer inequalities in \NODR). 
Indeed, (i) as shown in \cref{sec:dominance}, the relations $x_j =x_r$ derived by isomorphic aggregation are implied by the dominance inequalities; and (ii) as shown in \cref{sec_2link}, the dominance inequalities $x_j \leq x_r$ derived by dominance reduction are special cases of the two-customer inequalities\footnote{To further evaluate the individual performance of the proposed two-customer inequalities on the instances in testsets \DATAONE and \DATATWO, we have performed another experiment where only the two-customer inequality is implemented in \CPXC (denoted as \CPXCTCI). The results show that for instances in testset \DATAONE, the performance of \CPXCTCI is much better than that of \CPXC, while it is competitive to \IDI. For instances in testset \DATATWO, we can, however, observe a fairly large performance improvement of \IDI over \CPXCTCI; see Appendix G of the online supplement for more details.}.

The same argument can be applied in the context of solving the instances in testset \DATATWO where we only observe a slightly better gap improvement of \IDI over \NOAGG and \NODR.
However, for instances in testset \DATATWO, using the proposed isomorphic aggregation and dominance reduction, we can observe a fairly large reduction on the problem size; see columns $\RV$ and $\RC$ under setting $\NOAGG$ and column $\RC$ under setting $\NODR$.
Note that as the search space becomes smaller, this further leads to a reduction on the number of explored nodes; see \cref{T2:node1}.
Due to these improvements, the overall performance of \IDI is much better than \rev{that} of \NOAGG and \NODR. 
In particular, with the addition of the proposed isomorphic aggregation and dominance reduction, the CPU times are reduced by a factor of $6.41$ and $1.43$, respectively. 
In analogy to that on the instances in testset \DATAONE, the proposed two-customer inequalities have a significantly positive impact on the instances in testset \DATATWO. 
Overall, using the two-customer inequalities, $29$ more instances can be solved to optimality; and the CPU time and number of explored nodes are reduced by a factor of $3.02$ and $14.69$, respectively.

\begin{table}[!h]
		\centering
		\footnotesize
		\addtolength{\tabcolsep}{0pt}
		\renewcommand{\arraystretch}{1}
		\caption{Performance comparison of settings \NOAGG, \NODR, \NOVI, and \IDI.} \label{eachsettings}
		\begin{tabular}{c*{16}{r}}
			\toprule
			\multirow{2}{*}{Testsets} & \multicolumn{6}{c}{\NOAGG} & \multicolumn{5}{c}{\NODR} & \multicolumn{4}{c}{\NOVI} \\\cmidrule(r){2-7}\cmidrule(r){8-12}\cmidrule(r){13-16}
			& \multicolumn{1}{r}{\DeltaS} & \multicolumn{1}{r}{\RT} & \multicolumn{1}{r}{\RN} & \multicolumn{1}{r}{\RGPC} & \multicolumn{1}{r}{\RV} & \multicolumn{1}{r}{\RC} & \multicolumn{1}{r}{\DeltaS} & \multicolumn{1}{r}{\RT} & \multicolumn{1}{r}{\RN} & \multicolumn{1}{r}{\RGPC} & \multicolumn{1}{r}{\RC} & \multicolumn{1}{r}{\DeltaS} & \multicolumn{1}{r}{\RT} & \multicolumn{1}{r}{\RN} & \multicolumn{1}{r}{\RGPC} \\
			\hline
			\DATAONE &                0 & 1.01 & 1.00 & 0.00 & 1.03 & 1.06          & 1 & 0.96 & 1.00 & 0.00 & 1.09          & -24 & 3.04 & 21.25 & -46.32   \\ 
			\DATATWO &                -34 & 6.41 & 1.97 & -0.36 & 3.22 & 5.11          & 0 & 1.43 & 1.24 & -0.18 & 1.43          & -29 & 3.02 & 14.69 & -5.88   \\ 
			\bottomrule
		\end{tabular}
\end{table}

\begin{figure}[!h]
	\centering 
	\subfloat[\DATAONE]{\label{T1:time1}\includegraphics[width=0.47\textwidth]{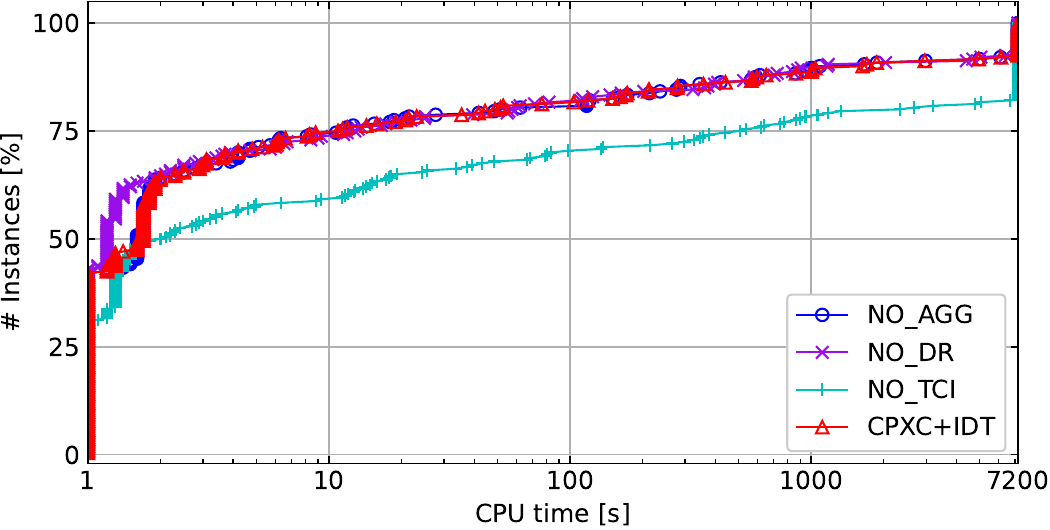}}
	\qquad
	\subfloat[\DATAONE]{\label{T1:node1}\includegraphics[width=0.47\textwidth]{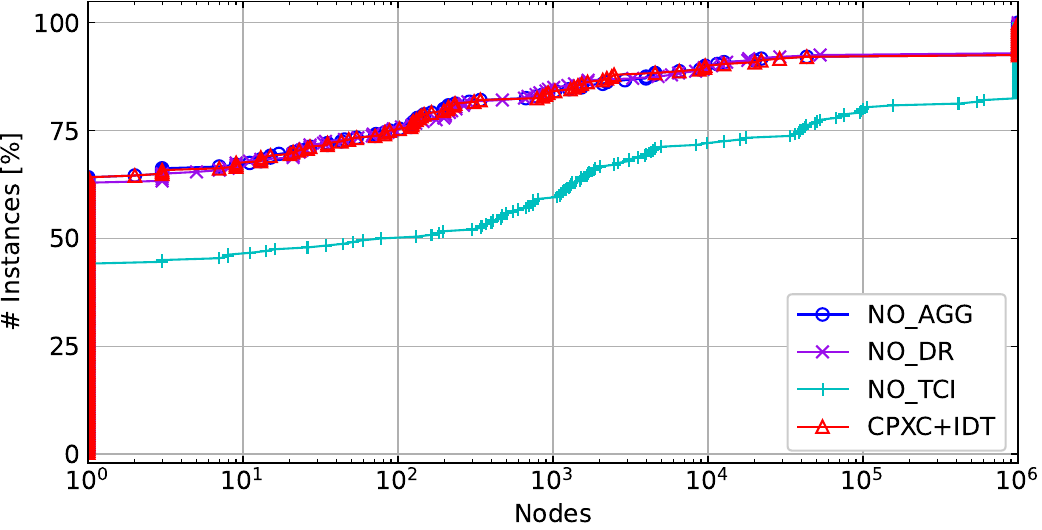}}
	\qquad
	\subfloat[\DATATWO]{\label{T2:time1}\includegraphics[width=0.47\textwidth]{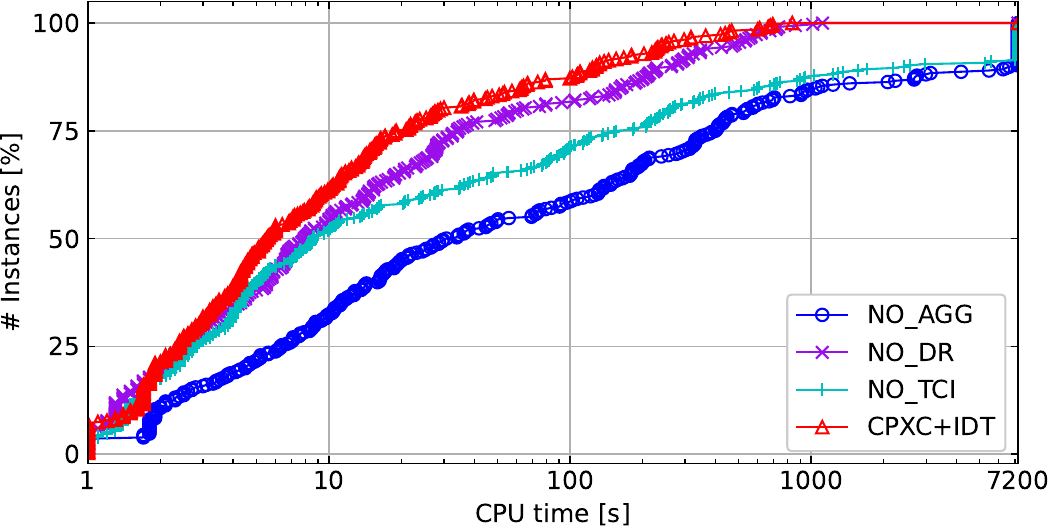}}
	\qquad
	\subfloat[\DATATWO]{\label{T2:node1}\includegraphics[width=0.47\textwidth]{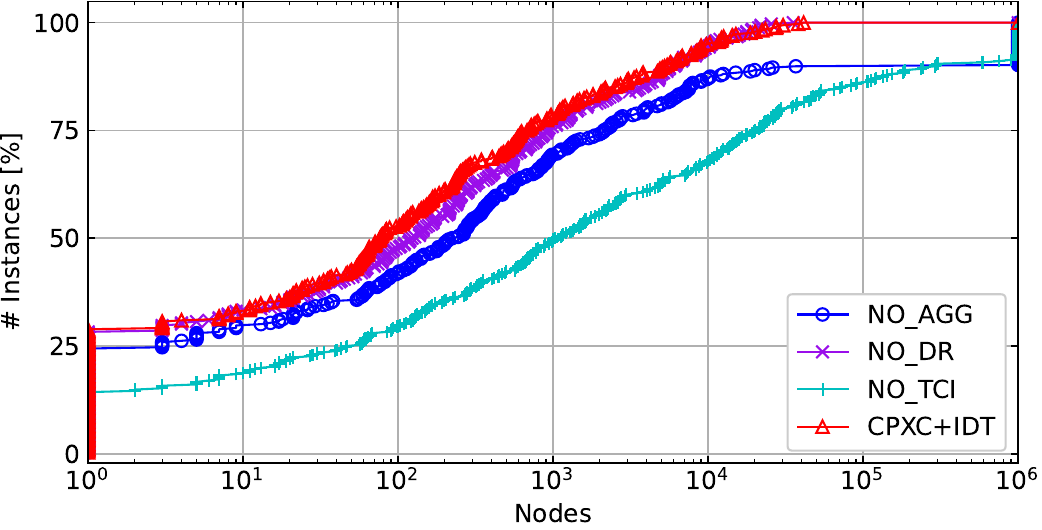}}
	\vspace*{-0.2cm}
	\caption{Performance profiles of the CPU time and number of explored nodes for settings \NOAGG, \NODR, \NOVI, and \IDI.}
	\vspace*{-0.2cm}
	\label{fig:eachtechnique}
\end{figure}

\section{Conclusion}\label{sec:conclusion}
In this paper, we have considered the \MCLPNW, where customer weights are allowed to be positive or negative, and proposed customized presolving and cutting plane techniques (namely, isomorphic aggregation, dominance reduction, and two-customer inequalities) to improve the computational performance of \MIP-based approaches. 
The proposed isomorphic aggregation and dominance reduction are able to not only reduce the problem size of the \MCLPNW but also improve the \LP relaxation of the problem formulation. 
The two-customer inequalities can be embedded into a branch-and-cut framework to further strengthen the \LP relaxation of the \MIP formulation on the fly. 
By extensive computational experiments, we have demonstrated that the three proposed techniques can substantially enhance the capability of \MIP solvers in solving \MCLPNW{s}.
In particular, the three proposed techniques enable us to turn many \MCLPNW instances from intractable to easily solvable.\\[4pt]
{\noindent\bf Acknowledgments}\\[4pt]
We would like to thank the three anonymous reviewers for their insightful comments that significantly improved the quality of our work.

\bibliography{shorttitles,GMCLP_arXiv}

\end{document}